\documentclass[10pt]{amsart}
\usepackage{amssymb,amsmath,amsthm}
\usepackage{graphicx}
\usepackage{subfigure}
\usepackage{color}
\newtheorem{theorem}{Theorem}[section]
\newtheorem{lemma}[theorem]{Lemma}
\newtheorem{corollary}{Corollary}[section]

\theoremstyle{definition}

\theoremstyle{remark}
\newtheorem{remark}[theorem]{Remark}

\makeatletter
\@namedef{subjclassname@2020}{2020 Mathematics Subject Classification}
\makeatother

\numberwithin{equation}{section}

\date{\today}

\begin{document}
%\today

\title[Instability of periodic waves from fold-Hopf bifurcation]
{Instability of small-amplitude periodic waves  from\\ fold-Hopf bifurcation}
%\thanks{This work was partly supported by the Hubei provincial postdoctoral science and technology activity project.}

\maketitle

%\author{Shuang Chen}
%\address{School of Mathematics and Statistics, Huazhong University of Science and Technology, Wuhan, Hubei 430074, China}
%\curraddr{ }
%\email{schen@hust.edu.cn}
\vskip 1pt
% Enter the first author's name and address:
\centerline{\scshape Shuang Chen}
\medskip
{\footnotesize
% please put the address of the first author
 \centerline{ School of Mathematics and Statistics, Central China Normal University}
   \centerline{ Wuhan, Hubei 430079, P. R. China}
   \centerline{ and }
   \centerline{ Center for Mathematical Sciences, Huazhong University of Sciences and Technology}
   \centerline{ Wuhan, Hubei 430074, P. R. China}
   \centerline{{\rm{Email}: schen@ccnu.edu.cn}}
} % Do not forget to end the {\footnotesize by the sign }

\medskip

\centerline{\scshape Jinqiao Duan\footnote{The corresponding author}}
\medskip

{\footnotesize
% please put the address of the first author
 \centerline{Department of Applied Mathematics, Illinois Institute of Technology}
   \centerline{Chicago, IL 60616, USA}
    \centerline{{\rm{Email}: duan@iit.edu}}
} % Do not forget to end the {\footnotesize by the sign }

\medskip

\begin{abstract}
We study the existence and stability of small-amplitude periodic waves emerging from fold-Hopf equilibria
in a system of one reaction-diffusion equation coupled with one ordinary differential equation.
This coupled system includes the FitzHugh-Nagumo system, caricature calcium models and other models in the real-world applications.
Based on the recent results on the averaging theory,
we solve periodic solutions in related three-dimensional systems
and then prove the existence of  periodic waves  arising from fold-Hopf bifurcations.
Numerical computation in [J. Tsai, W. Zhang, V. Kirk, and J. Sneyd, SIAM J. Appl. Dyn. Syst. 11 (2012), 1149--1199]
once suggested that the periodic waves from fold-Hopf bifurcations in a caricature calcium model
are spectrally unstable, yet without a proof.
After analyzing the linearization about periodic waves by the relatively bounded perturbation,
we prove the instability of small-amplitude periodic waves through
a perturbation of  the unstable spectra for the linearizations about the fold-Hopf equilibria.
As an application,
we prove the existence and stability of small-amplitude periodic waves from fold-Hopf bifurcations
in the FitzHugh-Nagumo system with an applied current.

\end{abstract}

\section{Introduction}
\label{sec-intr}

The excitable systems including the classical FitzHugh-Nagumo system \cite{FitzHugh-60,Nagumo-62}
have many different types of nonlinear waves,
such as asymptotically constant structures (e.g., pulses and fronts), spatially-periodic structures (e.g., wave trains).
These structures present wave propagation in excitable systems,
for example, traveling waves in  FitzHugh-Nagumo  system \cite{FitzHugh-60,Nagumo-62} indicate the propagation of nerve impulses in axons,
and in the caricature calcium models \cite{Keener-Sneyd-98} describe the changes of the free cytoplasmic calcium concentration and the
free calcium concentration in the endoplasmic reticulum.
In order to understand wave propagation well,
the study of nonlinear waves in the excitable systems has attracted much attention in the last decades.

As an important example of excitable systems,
the FitzHugh-Nagumo  system is a simplification of
the Hodgkin-Huxley model \cite{Hodgkin-Huxley-52} and has the form
\begin{eqnarray} \label{FHN-PDE-0}
\begin{aligned}
\frac{\partial u}{\partial t} &=d u_{xx}+h(u)-w+p,\\
\frac{\partial w}{\partial t} &=\delta(u-\gamma w),
\end{aligned}
\end{eqnarray}
where $u(x,t)$  is   the plasma membrane electric potential
and $w(x,t)$ indicates the combined inactivation effects of potassium and sodium ion channels.
We refer to Section \ref{sec-app} for more details on this model.
Since the FitzHugh-Nagumo system was proposed by FitzHugh \cite{FitzHugh-60} and Nagumo, Arimoto and Yoshizawa \cite{Nagumo-62},
many efforts have been devoted to understanding complex oscillations and traveling waves in this system
in the past tens of years.
See, for example, \cite{Cai-etal-19,Carter-Sandstede-15, Champneys-etal07,Chen-Hu-14,Cornwell-Jones-18,Czechowski-Piotr-16,Guck-Kuehn-09,Guck-Kuehn-10,Jones-89,Sch-Hupkes-19}
and references therein.

However,  not all nonlinear waves in biological systems can
be well understood by the FitzHugh-Nagumo system,
and many excitable systems also exhibit  different mechanisms giving  rise to spatially traveling waves.
More recently,
Tsai, Zhang, Kirk and Sneyd \cite{Tsai-etal-2012} studied another physiological excitable system of the form
\begin{eqnarray} \label{calcium-pde-0}
\begin{aligned}
\frac{\partial u}{\partial t} &=D u_{xx}+F(u)w-G(u),\\
\frac{\partial w}{\partial t} &=-\gamma (F(u)w-G(u)),
\end{aligned}
\end{eqnarray}
where
\begin{eqnarray*}
F(u)=\alpha+k\frac{u^{2}}{u^{2}+\varphi_{1}^{2}}\cdot \frac{\varphi_{2}}{u+\varphi_{2}},
\ \ \ \
G(u)=F(u)u+k_{s}u.
\end{eqnarray*}
This system is  a simplified model of calcium dynamics. Here
the system states $u(x,t)$ and $w(x,t)$ are the nondimensional concentration of
free cytoplasmic calcium and the nondimensional concentration of free calcium in the
endoplasmic reticulum, respectively.
We refer to \cite{Tsai-etal-2012} for the biological descriptions of the model parameters.

By applying the dynamical system approach and the Evans function,
Tsai, Zhang, Kirk and Sneyd \cite{Tsai-etal-2012} investigated the existence and stability of fronts and pulses.
After that, they  considered the conditions under which
periodic traveling  wave solutions (e.g., wave trains) arise from fold-Hopf equilibria,
at which the Jacobian matrix of the three-dimensional traveling wave system have one zero eigenvalue
and a pair of purely imaginary eigenvalues (see \cite{Guck-Holmes-83,Kuznetsov-98,Wiggins-03} or Section \ref{sec-pf-thm-exist} below).
They also numerically found that these small-amplitude periodic waves are spectrally unstable,
yet without providing a mathematically rigorous proof.
Furthermore, they found that the fold-Hopf bifurcation
(also called the Hopf-zero or Gavrilov-Guckenheimer bifurcation \cite{Guck-Holmes-83,Kuznetsov-98,Wiggins-03})
  is always subcritical for the calcium model \eqref{calcium-pde-0},
while for the FitzHugh-Nagumo system \eqref{FHN-PDE-0} the fold-Hopf bifurcation may be supercritical or subcritical.
This indicates that the FitzHugh-Nagumo system \eqref{FHN-PDE-0} and the simplified calcium model \eqref{calcium-pde-0}
exhibit  different mechanisms producing small-amplitude periodic traveling wave solutions.

The main contribution of this paper is to study the existence
and stability of periodic traveling wave solutions emerging from the fold-Hopf bifurcation.
More precisely, we investigate the small-amplitude periodic traveling waves from fold-Hopf bifurcations
in a general system of one reaction-diffusion equation coupled
with one ordinary differential equation
\begin{eqnarray} \label{PDE-ODE}
\begin{aligned}
\frac{\partial u}{\partial t} &=u_{xx}+f(u,w,\alpha),\\
\frac{\partial w}{\partial t} &=g(u,w,\alpha).
\end{aligned}
\end{eqnarray}
where $u(x,t)$ and $w(x,t)$ are system states, $x$ is a one-dimensional spatial variable, $t$ is time,
the parameter vector $\alpha$ is in $\mathbb{R}^{m}$ with $m\geq 1$,
and $f$ and $g$ are sufficiently smooth functions.
When the diffusion rate of $u$ in \eqref{PDE-ODE}  is a nonzero constant instead of one
(for example, in \eqref{FHN-PDE-0} and \eqref{calcium-pde-0}),
we can change it to be one by a simple rescaling.
The systems of the form \eqref{PDE-ODE} are widely used to understand the mechanisms
for various phenomena in numerous experiences.
Besides the FitzHugh-Nagumo systems with or without an applied current
\cite{Champneys-etal07,FitzHugh-60,Nagumo-62} and
caricature calcium models \cite{Tsai-etal-2012,Zhang-Sneyd-Tsai-14},
this  coupled system \eqref{PDE-ODE} also includes  consumer-resource models \cite{He-Lam-19,Zhang-etal-17},
predator-prey systems \cite{Dunbar-86,Murray-02} and so on.

Concerning the existence of periodic traveling waves in the coupled system \eqref{PDE-ODE},
it can be proved by detecting periodic orbits in its traveling wave system,
which is a three-dimensional (3D for short) system of ordinary differential equations.
However, the classical Poincar\'e-Bendixson Theorem for determining periodic orbits in planar systems
is not  applicable for the 3D systems.
This causes a major obstacle to the existence of periodic traveling waves for the coupled system \eqref{PDE-ODE}.
Some efforts have been made to overcome it.
For example, periodic solutions can be constructed by
the geometric singular perturbation theory \cite{Carpenter-77,Carter-Arnd-18,Soto-01},
 homoclinic bifurcation \cite{Sandstede-Scheel-01} and
the topological  methods \cite{Hastings-74,Maginu-80}.
These works obtained large-amplitude periodic traveling wave solutions.
It is worth mentioning that small-amplitude periodic solutions can be obtained by
perturbing the fold-Hopf equilibria.
As we know, 3D systems near fold-Hopf equilibria
 could exhibit complex dynamical behaviors by various perturbations.
In order to detect periodic solutions bifurcating from fold-Hopf equilibria,
we adopt the averaging theory.
It is one of efficient methods to solve  periodic solutions in nonautonomous differential systems
\cite{Guck-Holmes-83,Sanders-etal-07}.
The recent works such as \cite{Candido-Novaes-17,Candido-Novaes-20,Llibre-etal-14} successfully
applied it to prove the existence of  periodic orbits for 3D systems.
See Section \ref{sec-averag} for more information on the averaging theory.

The spectral stability of a traveling wave is determined by the spectrum of the linearization about it.
Regarding pulses and fronts in the coupled system \eqref{PDE-ODE},
the essential spectra for their linearizations are obtained by
analyzing the corresponding asymptotic operators (see \cite{Henry1981}),
and the point spectra are determined by the zeros of  the Evans function \cite{Alexander-Gardner-Jones-90,Evans-74,Jones-84,Kapitula-Promislow-13,Sandstede-02}.
A key feature of the periodic case is that the coefficients of the  linearization about a periodic wave are periodic,
instead of being asymptotically constant in the cases of fronts and pulses.
Then the spectrum of the linearization about a periodic wave in systems of the form \eqref{PDE-ODE}
is determined by the Floquet multipliers for a three-dimensional periodic system with a spectral parameter
(see  Section \ref{sec-spect}).
Unfortunately, it is not easy to obtain the Floquet multipliers for linear systems with periodic coefficients
(see, for instance, \cite{Doelman-Sandstede-09,Eszter-99,Gardner-93,Maginu-80,Sandstede-Scheel-01}),
even for the quite simple two-dimensional system such as the equivalent system of  Hill's equation.
This causes a big obstacle to analytically proving the instability of periodic traveling waves from fold-Hopf bifurcation
in the coupled system \eqref{PDE-ODE} in the present paper.

By the fold-Hopf bifurcation theory \cite{Guck-Holmes-83,Kuznetsov-98},
we observe that small-amplitude periodic solutions collapse to a fold-Hopf equilibrium as
the system parameters tend to some fixed values.
So we expect that the related periodic traveling waves could inherit the stability of the homogeneous rest state.
In order to prove it,
we adopt a spectral perturbation analysis for the related linearized operator.
The similar idea has been successfully applied to study the stability of permanent structures in partial differential equations
and the spectral problems arising in quantum mechanics.
See, for example, \cite{Alvarez-Plaza-21,Alvarez-Murillo-Plaza-22,Hislop-Sigal,Kato-95,Kollar-etal-19,Tri-Bernard-Kollar-18}.
We  describe the strategy for the detailed proof as follows.
We first transform the spectral problem for a periodic wave into the Sturm-Liouville problem with periodic boundary condition.
This makes the spectrum discrete.
Then we decompose it into a linear operator with constant coefficients and  a perturbation.
The key step to prove the instability is to check that
the perturbation is relatively bounded with respect to the linear operator with constant coefficients,
which can be viewed as the linearization about the fold-Hopf equilibrium (the zero-amplitude case).
This process involves some complicated calculations on  the relative bound of
the perturbation term with respect to the constant operator.
After giving the relation between the resolvent of the perturbed operator and that of the constant operator,
we prove a spectral perturbation result. This result helps us to prove
the spectral instability of small-amplitude periodic waves by using
 the spectrum for the zero-amplitude case (see Theorem \ref{thm-stab-PTW}).

We would like to point out that we rigorously prove
the instability of small-amplitude periodic traveling wave solutions
bifurcating from fold-Hopf equilibria.
This coincides with the numerical result in \cite{Tsai-etal-2012}.
By the results in Theorem \ref{thm-stab-PTW},
we also find that the instability of perturbed periodic waves does not depend on
whether the corresponding fold-Hopf bifurcation is subcritical or supercritical.
Both of these two mechanisms produce unstable periodic traveling wave solutions in the couple system \eqref{PDE-ODE}.
Consequently,
although the FitzHugh-Nagumo system \eqref{FHN-PDE-0} and the simplified calcium model \eqref{calcium-pde-0}
could undergo different types of fold-Hopf bifurcation,
the perturbed periodic waves have the same stability.

This paper is organized as follows.
We first prove the existence of small-amplitude periodic traveling wave solutions emerging from fold-Hopf equilibria via
the averaging theory in Section \ref{sec-pf-thm-exist}.
Based on the results on relatively bounded perturbation,
we prove the spectral instability of small-amplitude periodic traveling wave solutions in Section \ref{sec-instability}.
As an application,
we study the existence and stability of small-amplitude periodic traveling waves for
the FitzHugh-Nagumo system with an applied current in Section \ref{sec-app}.
We end with some concluding remarks in the final section.

\section{Existence of periodic traveling waves via averaging theory}
\label{sec-pf-thm-exist}

In this section, we study the existence of periodic traveling wave solutions bifurcating from fold-Hopf equilibria.
Before that, we first introduce some notions.
We call a solution $(u(x,t),w(x,t))$  of \eqref{PDE-ODE}  a \textit{traveling wave}
if there exists a real number $c$ and the function
$(\phi(\cdot),\psi(\cdot)): \mathbb{R}\to\mathbb{R}^{2}$ such that
\begin{eqnarray*}
(u(x,t),w(x,t))=(\phi(x+ct),\psi(x+ct)),
\end{eqnarray*}
where $c$ is the wave speed and $(\phi(\cdot),\psi(\cdot))$ is the wave profile.
Additionally, if $(\phi(\cdot),\psi(\cdot))$ is a periodic function with minimal period $A>0$,
that is,
\begin{eqnarray*}
(\phi(\xi+A),\psi(\xi+A))=(\phi(\xi),\psi(\xi)), \ \ \  \xi\in\mathbb{R},
\end{eqnarray*}
then $(u(x,t),w(x,t))=(\phi(x+ct),\psi(x+ct))$ is called a {\it periodic traveling wave} of \eqref{PDE-ODE}.
In case $c=0$ the traveling wave is called a {\it standing wave}.
Throughout this paper we focus on small-amplitude traveling waves
with nonzero wave speed in \eqref{PDE-ODE}.
In the moving coordinate frame $(\xi,t)=(x+ct,t)$,
a wave profile $(\phi(\cdot),\psi(\cdot))$ is a steady state of the following system
\begin{eqnarray} \label{eq-TW}
\begin{aligned}
u_{t}&=-cu_{\xi}+u_{\xi\xi}+f(u,w,\alpha),\\
w_{t}&=-cw_{\xi}+g(u,w,\alpha).
\end{aligned}
\end{eqnarray}
Set $v=u_{\xi}$. Then $(\phi,\phi_{\xi},\psi)$ is a solution of
a three-dimensional system of ordinary differential equations
\begin{eqnarray} \label{3D-SYSTEM}
\begin{aligned}
\frac{d u}{d \xi} &= \dot u  = v,
\\
\frac{d v}{d \xi} &= \dot v  = c v-f(u,w,\alpha),
\\
\frac{d w}{d \xi} &= \dot w  = \frac{1}{c} g(u,w,\alpha).
\end{aligned}
\end{eqnarray}
As a result, in order to prove the existence of periodic waves in the coupled system \eqref{PDE-ODE},
we only need to detect the periodic solutions in the 3D system \eqref{3D-SYSTEM}.

\subsection{Fold-Hopf equilibrium}
In this section, we recall that  an equilibrium of  the 3D system \eqref{3D-SYSTEM} is said to be
a {\it fold-Hopf equilibrium} if the Jacobian matrix for  the 3D system \eqref{3D-SYSTEM} at this point
has one zero eigenvalue $\lambda_{1}=0$ and a pair of purely imaginary eigenvalues $\lambda_{2,3}=\pm\, {\bf i}\, \mu_{0}$
with $\mu_{0}>0$ (see \cite[p.330]{Kuznetsov-98}),
where ${\bf i}=\sqrt{-1}$.
By perturbing fold-Hopf equilibria,
the 3D system \eqref{3D-SYSTEM} could undergo fold-Hopf bifurcations.
We refer to \cite{Guck-Holmes-83,Kuznetsov-98} for more information on the fold-Hopf bifurcation and fold-Hopf equilibrium.

In what follows,
we always use $^{T}$ to denote the transpose of a vector in the usual Euclidean inner product.
As a preparation, we first provide some conditions under which
the 3D system \eqref{3D-SYSTEM} has a fold-Hopf equilibrium.
More precisely, we have the following results.

\begin{lemma} \label{lm-FH-pt}
Assume that  $P_{0}:=(u_{0},v_{0},w_{0})^{T}$ in $\mathbb{R}^{3}$  is a fold-Hopf equilibrium
of \eqref{3D-SYSTEM} with $(\alpha,c)=(\alpha_{0},c_{0})$ and $c_{0}\neq 0$.
Let $\lambda_{1}=0$ and $\lambda_{2,3}=\pm\, {\bf i}\, \mu_{0}$ for some $\mu_{0}>0$
 be the eigenvalues of the Jacobian matrix $J(u_{0},v_{0},w_{0})$ at $P_{0}$.
Then the following conditions hold:
\begin{eqnarray}
&&v_{0}=0,\ \ \ f(u_{0},w_{0},\alpha_{0})=0,\ \ \ g(u_{0},w_{0},\alpha_{0})=0. \label{cond-1}\\
&&c_{0}+\frac{1}{c_{0}}g_{w}(P_{0})=f_{u}(P_{0})g_{w}(P_{0})-f_{w}(P_{0})g_{u}(P_{0})=0,\ \ \
f_{u}(P_{0})+g_{w}(P_{0})=\mu_{0}^{2}>0. \label{cond-2}
\end{eqnarray}
\end{lemma}
\begin{proof}
Since $P_{0}=(u_{0},v_{0},w_{0})^{T}$ is a fold-Hopf equilibrium  of \eqref{3D-SYSTEM} with $(\alpha,c)=(\alpha_{0},c_{0})$,
the condition \eqref{cond-1} holds.
It is clear that the Jacobian matrix $J(u_{0},v_{0},w_{0})$ of \eqref{3D-SYSTEM} at this point is in the form
\begin{equation*}
J(u_{0},v_{0},w_{0})=\left(
\begin{array}{ccc}
0 & 1 & 0\\
-f_{u}(P_{0}) & c_{0} & -f_{w}(P_{0})\\
\frac{1}{c_{0}}g_{u}(P_{0}) & 0 & \frac{1}{c_{0}}g_{w}(P_{0})
\end{array}
\right),
\end{equation*}
and the eigenvalues  of $J(u_{0},v_{0},w_{0})$ are determined by the zeros of its  characteristic polynomial
\begin{eqnarray*}
&&\lefteqn{{\rm det}(\lambda I_{3}-J(u_{0},v_{0},w_{0}))}\\
&&=\lambda^{3}-\left(c_{0}+\frac{1}{c_{0}}g_{w}(P_{0})\right)\lambda^{2}
   +\left(f_{u}(P_{0})+g_{w}(P_{0})\right)\lambda +\frac{1}{c_{0}}\left(f_{w}(P_{0})g_{u}(P_{0})-f_{u}(P_{0})g_{w}(P_{0})\right),
\end{eqnarray*}
where $I_{3}$ is the $3\times 3$ identity matrix.
By the definition of the fold-Hopf equilibrium,
we have
\begin{eqnarray*}
c_{0}+\frac{1}{c_{0}}g_{w}(P_{0})\!\!\!&=&\!\!\!\lambda_{1}+\lambda_{2}+\lambda_{3}=0,\\
f_{u}(P_{0})+g_{w}(P_{0})\!\!\!&=&\!\!\lambda_{1}\lambda_{2}+\lambda_{1}\lambda_{3}+\lambda_{2}\lambda_{3}=\mu_{0}^{2}>0, \\
\frac{1}{c_{0}}\left(f_{u}(P_{0})g_{w}(P_{0})-f_{w}(P_{0})g_{u}(P_{0})\right)\!\!\!&=&\!\!\!\lambda_{1}\lambda_{2}\lambda_{3}=0.
\end{eqnarray*}
Thus, \eqref{cond-2} holds. This finishes the proof.
\end{proof}

\subsection{Averaging theory}
\label{sec-averag}

In this section, we take advantage of the averaging theory to find
periodic solutions bifurcating from fold-Hopf equilibria in the 3D system \eqref{3D-SYSTEM}.
The averaging theory is one of the efficient methods to find periodic solutions in nonautomous differential systems
(see more details in \cite{Candido-Novaes-17,Guck-Holmes-83,Llibre-etal-14,Sanders-etal-07}).
In the following we introduce some results on the averaging theory.
Consider a  nonautonomous differential system in the form
\begin{eqnarray} \label{eq-average}
\dot X(t)=\epsilon F(t,X)+\epsilon^{2}K(t,X,\epsilon),
\end{eqnarray}
where $X=(x,y)^{T}$ is in $\mathbb{R}^{2}$, small parameter $\epsilon$ satisfies $0\leq \epsilon<|\epsilon_{1}|$
for a sufficiently small $\epsilon_{1}>0$,
and the functions $F$ and $K$ are sufficiently smooth and are $A$-periodic in the variable $t\in \mathbb{R}$.
Then \eqref{eq-average} can be defined in the extended phase space $\mathbb{S}^{1}\times \mathbb{R}^{2}$ by
setting $\dot t=1$ mod($A$).
Let $X(t,Z,\epsilon)$ denote the solution of \eqref{eq-average} with the initial value $X(0)=Z$ in $\mathbb{R}^{2}$.
Then by \cite[Lemma 5, p. 3570]{Candido-Novaes-17},
the solution  $X(t,Z,\epsilon)$ has the expansion
\begin{eqnarray*}
X(t,Z,\epsilon)=Z+\epsilon Y(t,Z)+O(\epsilon^{2}),
\end{eqnarray*}
where $Y$ is given by
\begin{eqnarray*}
Y(t,Z)=\int_{0}^{t}F(\tau,Z)d\tau.
\end{eqnarray*}
Consequently, the Poincar\'e map $\Pi(Z,\epsilon)=X(A,Z,\epsilon)$ is in the form
\begin{eqnarray*}
\Pi(Z,\epsilon)=Z+\epsilon G(Z)+O(\epsilon^{2}),
\end{eqnarray*}
where $G(Z)$ is given by
\begin{eqnarray*}
G(Z)=Y(A,Z)=\int_{0}^{A}F(\tau,Z)d\tau,
\end{eqnarray*}
and is called the {\it  averaged function of order one}.
Hence, in order to detect the periodic solutions in \eqref{eq-average},
we turn to solve the equation $\Pi(Z,\epsilon)=Z$.
By \cite[Theorem A, p. 566]{Llibre-etal-14} we have the next lemma.

\begin{lemma} \label{lm-averag-1}
Assume that $G\neq 0$, and there exists a point $Z_{0}$ in $\mathbb{R}^{2}$ such that
$G(Z_{0})=0$ and the differential  $DG(Z_{0})$ of $G(Z)$ is an invertible  matrix.
Then for a sufficiently small $|\epsilon|>0$,
the perturbed system \eqref{eq-average} has an isolated $A$-periodic solution $\varphi(t,\epsilon)$
satisfying that $\varphi(0,\epsilon)\to Z_{0}$ as $\epsilon\to0$.
\end{lemma}

In the preceding lemma, we apply the averaged function of order one to find the periodic solutions in \eqref{eq-average}.
We also refer to \cite{Candido-Novaes-17,Candido-Novaes-20,Llibre-etal-14} for more recent results
on the applications of the averaged function of high order.

\subsection{Existence of periodic traveling waves}

Now we study the existence of small-amplitude periodic traveling waves emerging from
fold-Hopf equilibria via the averaging theory.

Without loss of generality,
we assume that for $(\alpha,c)=(\alpha_{0},c_{0})$,
the 3D system \eqref{3D-SYSTEM} has a fold-Hopf equilibrium at the origin, which is denoted by $0$ for short.
Let $X(u,w,v,\epsilon)$ denote the vector field defined  by \eqref{3D-SYSTEM} with $(\alpha,c)$ in the form
\begin{eqnarray} \label{eq-series}
\alpha=\alpha_{0}+\alpha_{1}(\epsilon):=\alpha_{0}+\sum_{k=1}^{+\infty}\tilde{\alpha}_{k}\epsilon^{k}, \ \ \ \
c=c_{0}+c_{1}(\epsilon):=c_{0}+\sum_{k=1}^{+\infty}\tilde{c}_{k}\epsilon^{k},
\end{eqnarray}
where  all coefficients will be fixed according to our need such that
the radii of convergence of these two series are nonzero.
It is clear that by a change
\begin{eqnarray} \label{df-change}
(u,v,w)^{T}\to \epsilon Q (u,v,w)^{T},\ \ \
Q=\left(
\begin{array}{ccc}
-\frac{1}{c_{0}}g_{w}(0,\alpha_{0}) & \mu_{0} & g_{w}(0,\alpha_{0})\\
-\mu_{0}^{2}  & c_{0}\mu_{0} & 0\\
\frac{1}{c_{0}}g_{u}(0,\alpha_{0}) & 0 & -g_{u}(0,\alpha_{0})
\end{array}
\right),
\end{eqnarray}
where  the matrix $Q$ consists of  the eigenvectors corresponding to
the Jacobian matrix  of \eqref{3D-SYSTEM} at the origin,
the 3D system \eqref{3D-SYSTEM} is transformed into
\begin{eqnarray} \label{3D-1}
\begin{aligned}
\dot u  &= \mu_{0} v+X_{1}(u,v,w,\epsilon),
\\
\dot v  &= -\mu_{0} u+X_{2}(u,v,w,\epsilon),
\\
\dot w  &= X_{3}(u,v,w,\epsilon),
\end{aligned}
\end{eqnarray}
where
\begin{eqnarray*}
\begin{aligned}
&(X_{1}(u,v,w,\epsilon),X_{2}(u,v,w,\epsilon),X_{3}(u,v,w,\epsilon))^{T}\\
&=\epsilon^{-1}Q^{-1}X(\epsilon (u,v,w)Q^{T},\epsilon)-(\mu_{0}v,-\mu_{0}u,0)^{T}.
\end{aligned}
\end{eqnarray*}
For sufficiently small $|\epsilon|$,
the expansion of $(X_{1}(u,v,w,\epsilon),X_{2}(u,v,w,\epsilon),X_{3}(u,v,w,\epsilon))^{T}$
with respect to $\epsilon$ can be written as
\begin{eqnarray*}
\begin{aligned}
&(X_{1}(u,v,w,\epsilon),X_{2}(u,v,w,\epsilon),X_{3}(u,v,w,\epsilon))^{T}\\
&=(X_{10}+\epsilon X_{11}(u,v,w),X_{20}+\epsilon X_{21}(u,v,w),X_{30}+\epsilon X_{31}(u,v,w))^{T}+O(\epsilon^{2}),
\end{aligned}
\end{eqnarray*}
where $X_{j1}(u,v,w)=O(|(u,v,w)|)$ for sufficiently small $|(u,v,w)|$.
Note that $X_{j0}$ and $X_{j1}$ depend on the matrix $Q$, which is not unique.
Then for simplicity, we omit their lengthy expressions, which are
determined by the first and second partial derivatives of the functions $f$ and $g$ with respect to $u$, $v$, $w$ and $\alpha$.
As a concrete example, we give the explicit expressions of $X_{j0}$ and $X_{j1}$
for the FitzHugh-Nagumo system in Section \ref{sec-exmp-thm}.
Consider system \eqref{3D-1} in the cylindrical coordinate
\begin{eqnarray*}
(u,v,w)=(r\cos\theta,r\sin\theta,w), \ \ \ \ \ r\geq 0,\ \ \ 2\pi \geq \theta\geq 0, \ \ \ w\in\mathbb{R}.
\end{eqnarray*}
Then  system \eqref{3D-1} is converted into
\begin{eqnarray} \label{3D-2}
\begin{aligned}
\dot r  &= X_{1} \cos\theta+X_{2}\sin\theta,
\\
r\dot \theta  &= -\mu_{0} r-X_{1} \sin\theta+X_{2}\cos\theta,
\\
\dot w  &= X_{3},
\end{aligned}
\end{eqnarray}
where we  write $X_{j}$  with $(r\cos\theta,r\sin\theta,w,\epsilon)$ being suppressed  for simplicity.
Then for sufficiently small $|r|$ and $|\epsilon|$,
\begin{eqnarray} \label{2D-1}
\begin{aligned}
\frac{d r}{d\theta} &=
      \frac{r(X_{1} \cos\theta+X_{2}\sin\theta)}{-\mu_{0} r-X_{1} \sin\theta+X_{2}\cos\theta},
\\
\frac{d w}{d\theta}  &=\frac{rX_{3}}{-\mu_{0} r-X_{1} \sin\theta+X_{2}\cos\theta}.
\end{aligned}
\end{eqnarray}
Clearly, the right-hand side of  system \eqref{2D-1} is $2\pi$-periodic in the variable $\theta\in \mathbb{R}$.
Then we have the following results.

\begin{theorem} \label{thm-existence} {\rm (Existence of periodic traveling waves)}
Define two functions $R_{1}$ and $R_{2}$ by
\begin{eqnarray*}
R_{1}(r,w)\!\!\!&=&\!\!\!\int_{0}^{2\pi}\left( \cos\theta\cdot X_{11}(r\cos\theta,r\sin\theta,w)
                 +\sin\theta\cdot X_{21}(r\cos\theta,r\sin\theta,w)\right) d\theta, \\
R_{2}(r,w)\!\!\!&=&\!\!\!\int_{0}^{2\pi}  X_{31}(r\cos\theta,r\sin\theta,w)\,d\theta.
\end{eqnarray*}
Assume that there exist $\alpha_{1}(\epsilon)$ and $c_{1}(\epsilon)$ defined by \eqref{eq-series}
such that $X_{10}=X_{20}=X_{30}=0$, $R_{1}(r_{*},w_{*})=R_{2}(r_{*},w_{*})=0$,
and
\begin{eqnarray} \label{eq-inv}
{\rm det}\left(
\begin{aligned}
\frac{\partial R_{1}}{\partial r} & \frac{\partial R_{1}}{\partial w}\\
\frac{\partial R_{2}}{\partial r} & \frac{\partial R_{2}}{\partial w}
\end{aligned}
\right)|_{(r_{*},w_{*})}\neq 0,
\end{eqnarray}
for some $r_{*}>0$ and $w_{*}\in\mathbb{R}$.
Then  there exists  a sufficiently small $\epsilon_{0}>0$
such that  a small-amplitude periodic wave $(\phi,\psi)$ with period
$A_{\epsilon}$ and wave speed $c_{0}+c_{1}(\epsilon)$ emerges from the origin
in the coupled system \eqref{PDE-ODE} with $\alpha=\alpha_{0}+\alpha_{1}(\epsilon)$
for each $\epsilon\in(0,\epsilon_{0}]$,
where
\begin{eqnarray*}
|(\phi,\psi)|=O(\epsilon),\ \ \ \ A_{\epsilon}=2\pi/\mu_{0}+O(\epsilon), \ \ \ \ \epsilon\in(0,\epsilon_{0}].
\end{eqnarray*}
\end{theorem}
\begin{proof}
Note that  the periodic traveling waves of the coupled system \eqref{PDE-ODE}
correspond to the periodic solutions of the 3D system \eqref{3D-SYSTEM}.
Then it suffices to prove that
there exists  a sufficiently small $\epsilon_{0}>0$
such that  a small-amplitude periodic solution $(\phi,\phi_{\xi},\psi)$ with period
$A_{\epsilon}$ bifurcates from the origin
in the 3D system \eqref{3D-SYSTEM} with $(\alpha,c)=(\alpha_{0}+\alpha_{1}(\epsilon),c_{0}+c_{1}(\epsilon))$
for each $\epsilon\in(0,\epsilon_{0}]$,
where $|(\phi,\phi_{\xi},\psi)|=O(\epsilon)$ and $A_{\epsilon}=2\pi/\mu_{0}+O(\epsilon)$ for $\epsilon\in(0,\epsilon_{0}]$.

Since $X_{10}=X_{20}=X_{30}=0$, system \eqref{2D-1} has the expansion
\begin{eqnarray} \label{2D-2}
\begin{aligned}
\frac{d r}{d\theta} &= -\frac{\epsilon}{\mu_{0}}
      (\cos\theta X_{11}(r\cos\theta,r\sin\theta,w)+\sin\theta X_{21}(r\cos\theta,r\sin\theta,w))+O(\epsilon^{2}),
\\
\frac{d w}{d\theta}  &=-\frac{\epsilon}{\mu_{0}}X_{31}(r\cos\theta,r\sin\theta,w)+O(\epsilon^{2}).
\end{aligned}
\end{eqnarray}
Consequently,
by $R_{1}(r_{*},w_{*})=R_{2}(r_{*},w_{*})=0$, \eqref{eq-inv} and Lemma \ref{lm-averag-1},
we have that for sufficiently small $\epsilon$,
system \eqref{2D-2} has an isolated periodic solution $\varphi(\theta,\epsilon)$ of period $2\pi$ in $\theta$
such that $\varphi(0,\epsilon)\to (r_{*},w_{*})$ as $\epsilon\to 0$.
This together with \eqref{df-change} yields that
the 3D system \eqref{3D-SYSTEM} has a periodic solution $(\phi,\phi_{\xi},\psi)$
with $|(\phi,\phi_{\xi},\psi)|=O(\epsilon)$ for sufficiently small $|\epsilon|$.
By the second equation in \eqref{3D-2},
we have that for sufficiently small $|\epsilon|$,
\begin{eqnarray*}
A_{\epsilon}=|\int_{0}^{2\pi}\frac{r}{-\mu_{0} r-X_{1} \sin\theta+X_{2}\cos\theta}|_{(r,w)=\varphi(\theta,\epsilon)}\,d\theta|
=\frac{2\pi}{\mu_{0}}+O(\epsilon).
\end{eqnarray*}
Therefore, the proof is now complete.
\end{proof}

\section{Instability of periodic traveling waves via perturbation theory}
\label{sec-instability}

As for small-amplitude periodic traveling waves arising from the fold-Hopf bifurcations,
we are interested in the stability of this class of  periodic traveling waves in the coupled system \eqref{PDE-ODE}.
\subsection{Spectral theory for periodic traveling waves}
\label{sec-spect}

We first introduce basic notions on the spectral theory for
traveling waves in the coupled system \eqref{PDE-ODE}.
Assume that the coupled system \eqref{PDE-ODE} has a traveling wave $(\phi,\psi)$
with wave speed $c$.
The linearization of \eqref{eq-TW} about the traveling wave $(\phi,\psi)$ is in the form
\begin{eqnarray} \label{df-L}
\mathcal{L}
\left(
\begin{array}{c}
u \\
w
\end{array}
\right)
:= \left(
\begin{array}{c}
u_{\xi\xi}-cu_{\xi}+f_{u}(\phi,\psi,\alpha)u+f_{w}(\phi,\psi,\alpha)w \\
-cw_{\xi}+g_{u}(\phi,\psi,\alpha)u+g_{w}(\phi,\psi,\alpha)w
\end{array}
\right),
\end{eqnarray}
where
$$\mathcal{L}: H^{2}(\mathbb{R})\times H^{1}(\mathbb{R})\subset L^{2}(\mathbb{R})\times L^{2}(\mathbb{R})
\to L^{2}(\mathbb{R})\times L^{2}(\mathbb{R}).$$
The operator $\mathcal{L}$ is a closed operator on $L^{2}(\mathbb{R})\times L^{2}(\mathbb{R})$
with the domain $\mathcal{D}(\mathcal{L})=H^{2}(\mathbb{R})\times H^{1}(\mathbb{R})$ and
its eigenvalue problem is determined by
\begin{eqnarray*}
\mathcal{L}(u,w)^{T}=\lambda(u,w)^{T},
\end{eqnarray*}
where $(u,w)^{T}$ is in $\mathcal{D}(\mathcal{L})$, and $\lambda$ is in the complex plane $\mathbb{C}$.
Let $\rho(\mathcal{L})$ denote the {\it resolvent set} of $\mathcal{L}$,
consisting of all points $\lambda\in \mathbb{C}$ such that $\lambda I-\mathcal{L}$ is invertible
and $(\lambda I-\mathcal{L})^{-1}$ is a continuous linear operator,
where $I$ is the identity.
The complement of $\rho(\mathcal{L})$ in $\mathbb{C}$ is called the {\it spectrum} of  $\mathcal{L}$,
which is denoted by $\sigma(\mathcal{L})$.
Then
\begin{eqnarray*}
\mathbb{C}=\rho(\mathcal{L}) \cup \sigma(\mathcal{L}).
\end{eqnarray*}
If the dimension of the kernel ${\rm Ker}(\mathcal{L})$ for $\mathcal{L}$
and the codimension of the the range $\mathcal{R}(\mathcal{L})$ for $\mathcal{L}$
are both finite,
then the operator $\mathcal{L}$ is called a {\it Fredholm operator}  with the {\it Fredholm index}
$${\rm Ind}(\mathcal{L})={\rm dim}({\rm Ker}(\mathcal{L}))-{\rm codim}(\mathcal{R}(\mathcal{L})).$$
According to the definitions of the Fredholm operator and the Fredholm index,
the spectrum $\sigma(\mathcal{L})$ is classified into two different parts, that is,
$$\sigma(\mathcal{L})=\sigma_{\rm ess}(\mathcal{L}) \cup \sigma_{\rm pt}(\mathcal{L}),$$
where
$\sigma_{\rm ess}(\mathcal{L})$ (called the {\it essential spectrum})
consists of all $\lambda\in \mathbb{C}$ such that
$\lambda I-\mathcal{A}$ is not Fredholm or
$\lambda I-\mathcal{A}$ is Fredholm with ${\rm Ind}(\lambda I-\mathcal{L})\neq 0$,
and $\sigma_{\rm pt}(\mathcal{L})$ (called the {\it point spectrum})
consists of all $\lambda\in \mathbb{C}$ such that
$\lambda I-\mathcal{A}$  is not invertible and ${\rm Ind}(\lambda I-\mathcal{L})$ is zero.
Every $\lambda\in \sigma_{\rm pt}(\mathcal{L})$ is called an {\it eigenvalue} of $\mathcal{L}$ and
every nontrivial $(u,w)^{T}$ in $\mathcal{D}(\mathcal{L})$
satisfying $\mathcal{L}(u,w)^{T}=\lambda (u,w)^{T}$ is called
an {\it eigenfunction} corresponding to $\lambda$.

We next introduce the spectral projections
associated with $\sigma(\mathcal{L})$,
which are useful in the study of  the spectral perturbation problem (see Section \ref{sec-pert}).
Assume that  $\sigma(\mathcal{L})$ is separated into two disjoint parts $\sigma_{1}$ and $\sigma_{2}$
by a simple closed positively oriented cycle $\Gamma$,
which encloses a bounded open set containing  $\sigma_{1}$ in its interior and $\sigma_{2}$ in its exterior.
Then by \cite[Theorem 6.17, p.178]{Kato-95},
we can define the spectral projection $P_{\sigma_{1}}(\mathcal{L})$ associated with $\sigma_{1}$ by
the Dunford integral formula
\begin{eqnarray*}
P_{\sigma_{1}}(\mathcal{L})=\frac{1}{2\pi {\bf i}}\oint_{\Gamma}(\lambda I-\mathcal{L})\, d\lambda.
\end{eqnarray*}
In particular, if $\sigma_{1}$ only contains an isolated eigenvalue $\lambda_{0}$ of $\mathcal{L}$,
then the dimension of the range of $P_{\lambda_{0}}(\mathcal{L})$ is
called the {\it algebraic multiplicity of $\lambda_{0}$},
and the dimension of ${\rm Ker}(\lambda_{0} I-\mathcal{L}_{0})$
is called the {\it geometric multiplicity of $\lambda_{0}$}.
We refer to \cite{Kapitula-Promislow-13,Kato-95} for more details on the spectral theory of $\mathcal{L}$.
For notational convenience,
these notations are similarly defined for the linear operators which appear in the present paper.

Our goal is to study the stability of periodic traveling waves emerging from the fold-Hopf points.
For this reason, we give more details on the spectral theory for this class of periodic traveling waves.
Assume that $(\phi,\phi_{\xi},\psi)$ is a periodic solution with period $A_{\epsilon}$ arising from the fold-Hopf bifurcation
in the 3D system \eqref{3D-SYSTEM} with
$(\alpha,c)=(\alpha_{0}+\alpha_{1}(\epsilon),c_{0}+c_{1}(\epsilon))$ for $\epsilon\in(0,\epsilon_{0}]$,
where $\alpha_{1}$ and $c_{1}$ are defined as in \eqref{eq-series}.
Then $(\phi,\psi)$ is periodic wave  profile with period $A_{\epsilon}$ in system \eqref{PDE-ODE}.
So $(\phi(\cdot),\psi(\cdot))$ satisfies \eqref{eq-TW}.
The spectral stability of the profile $(\phi(\xi),\psi(\xi))$ is determined by the spectra of
the operator $\mathcal{L}$ in the form \eqref{df-L},
where all coefficients are periodic functions with period $A_{\epsilon}$.
Let $Y=(u,v,w)^{T}$ in $\mathbb{R}^{3}$.
Then the spectral problem  $\mathcal{L}(u,w)^{T}=\lambda (u,w)^{T}$ becomes a first-order system of ordinary differential equations
\begin{eqnarray} \label{eq-spect}
\frac{d Y}{d \xi}=A(\xi,\lambda)Y, \ \ \
A(\xi,\lambda)=\left(
\begin{array}{ccc}
0 & 1 & 0\\
\lambda-f_{u}(\phi,\psi,\alpha) & c & -f_{w}(\phi,\psi,\alpha)\\
\frac{1}{c}g_{u}(\phi,\psi,\alpha) & 0 & \frac{1}{c}(g_{w}(\phi,\psi,\alpha)-\lambda)
\end{array}
\right),
\end{eqnarray}
where $A(\xi+A_{\epsilon})=A(\xi,\lambda)$ for every $\lambda\in \mathbb{C}$.
We can check that the operator $\mathcal{L}$ has no point spectrum,
and $\lambda=\lambda_{0}$ is an essential spectrum of $\mathcal{L}$ if and only if
there exists a  $\mu$ in $(-\pi/A_{\epsilon},\pi/A_{\epsilon}]$ such that
\eqref{eq-spect} has a nonzero solution $Y$ satisfying
\begin{eqnarray*}
Y(A_{\epsilon})=e^{{\bf i}\mu A_{\epsilon}}Y(0).
\end{eqnarray*}
By the Bloch-wave decomposition
$$(u,w)\to (e^{{\bf i}\mu \xi}u,e^{{\bf i}\mu \xi}w),$$
the operator $\mathcal{L}$ is converted into
\begin{eqnarray} \label{L-mu}
\mathcal{L}_{\mu}
\left(
\begin{array}{c}
u \\
w
\end{array}
\right)
:= \left(
\begin{array}{c}
(\partial_{\xi}+{\bf i} \mu)^{2}u-c(\partial_{\xi}+{\bf i} \mu)u+f_{u}(\phi,\psi,\alpha)u+f_{w}(\phi,\psi,\alpha)w \\
-c(\partial_{\xi}+{\bf i} \mu)w+g_{u}(\phi,\psi,\alpha)u+g_{w}(\phi,\psi,\alpha)w
\end{array}
\right),
\end{eqnarray}
where
\begin{eqnarray*}
\mathcal{L}_{\mu}:
L^{2}_{\rm per}([0,A_{\epsilon}],\mathbb{C})\times L^{2}_{\rm per}([0,A_{\epsilon}],\mathbb{C})
\to L^{2}_{\rm per}([0,A_{\epsilon}],\mathbb{C})\times L^{2}_{\rm per}([0,A_{\epsilon}],\mathbb{C})
\end{eqnarray*}
with the domain $\mathcal{D}(\mathcal{L}_{\mu})=H^{2}_{\rm per}([0,A_{\epsilon}],\mathbb{C})\times H^{1}_{\rm per}([0,A_{\epsilon}],\mathbb{C})$.
Then  for each $\mu$ in  $(-\pi/A_{\epsilon},\pi/A_{\epsilon}]$,
the point spectrum of $\mathcal{L}_{\mu}$ is an essential spectrum of $\mathcal{L}$.
We refer to \cite[pp.\,68-70]{Kapitula-Promislow-13} for more details on
the relation between the spectra of $\mathcal{L}$ and those of $\mathcal{L}_{\mu}$ for $\mu$ in  $(-\pi/A_{\epsilon},\pi/A_{\epsilon}]$.
We summarize the above statements as follows:
\begin{lemma}
The spectrum $\sigma(\mathcal{L})$ of the operator $\mathcal{L}$
for periodic traveling wave $(\phi,\psi)$ in \eqref{df-L}
is in the form
\begin{eqnarray} \label{spect-decomp}
\sigma(\mathcal{L})=\bigcup_{-\frac{\pi}{A_{\epsilon}}<\mu\leq \frac{\pi}{A_{\epsilon}}}\sigma_{\rm pt}(\mathcal{L}_{\mu}).
\end{eqnarray}
\end{lemma}

In order to analyze the spectrum of $\mathcal{L}_{\mu}$,
we  also change the spectral problem for $\mathcal{L}_{\mu}$ into a first-order system of  ordinary differential equations
\begin{eqnarray} \label{eq-period}
\frac{d Y}{d \xi}=(A(\xi,\lambda)-{\bf i}\mu I_{3})Y, \ \ \  Y(A_{\epsilon})=Y(0),
\end{eqnarray}
where ${\bf i}=\sqrt{-1}$ and $I_{3}$ is the $3\times 3$ identity matrix.
Let $\Phi(\xi,\lambda,\mu)$ denote the fundamental matrix solution of \eqref{eq-period}.
Then \eqref{eq-period} has a nontrivial solution if and only if
\begin{eqnarray} \label{df-evans}
E(\lambda,\mu):={\rm det}(\Phi(A_{\epsilon},\lambda,\mu)-I_{3})=0, \ \ \  \lambda\in \mathbb{C},
\ \ \ -\frac{\pi}{A_{\epsilon}}<\mu\leq \frac{\pi}{A_{\epsilon}},
\end{eqnarray}
where $E(\lambda,\mu)$ is called the {\it Evans function}
(see, for instance,
\cite{Alexander-Gardner-Jones-90,Evans-74,Gardner-93,Jones-84,Kapitula-Promislow-13,Sandstede-02}).
Then the spectrum of $\mathcal{L}_{\mu}$ is determined by the zeros of $E(\lambda,\mu)$.

\subsection{Relatively bounded perturbation}
\label{sec-pert}

By \eqref{spect-decomp} we observe that
$\sigma(\mathcal{L})$ is spectrally unstable
if $\mathcal{L}_{\mu}$ is spectrally unstable for a certain $\mu$ in $(-\pi/A_{\epsilon},\pi/A_{\epsilon}]$.
Particularly, we consider the case $\mu=0$ in \eqref{L-mu}.
Then $\mathcal{L}_{\mu}$ is reduced to
$\mathcal{L}_{0}=\mathcal{L}$
with the domain $\mathcal{D}(\mathcal{L}_{0})=H^{2}_{\rm per}([0,A_{\epsilon}],\mathbb{C})\times H^{1}_{\rm per}([0,A_{\epsilon}],\mathbb{C})$.
In order to remove the $\epsilon$-dependence associated with the domain $\mathcal{D}(\mathcal{L}_{0})$,
we make a change
\begin{eqnarray*}
\xi\to A_{\epsilon}\xi, \ \ \ \ \ \epsilon\in(0,\epsilon_{0}].
\end{eqnarray*}
Then the spectral problem associated with the periodic wave $(\phi,\psi)$
is converted into the spectral problem for the following operator
\begin{eqnarray} \label{F-eps}
\ \ \ \ \mathcal{F}(\epsilon)
\left(
\begin{array}{c}
u \\
w
\end{array}
\right)
:= \left(
\begin{array}{c}
u_{\xi\xi}-cA_{\epsilon}u_{\xi}+A_{\epsilon}^{2}f_{u}(\phi,\psi,\alpha)u+A_{\epsilon}^{2}f_{w}(\phi,\psi,\alpha)w \\
-cA_{\epsilon}w_{\xi}+A_{\epsilon}^{2}g_{u}(\phi,\psi,\alpha)u+A_{\epsilon}^{2}g_{w}(\phi,\psi,\alpha)w
\end{array}
\right)
=\lambda
\left(
\begin{array}{c}
u \\
w
\end{array}
\right),
\end{eqnarray}
where
\begin{eqnarray*}
\mathcal{F}(\epsilon):
L^{2}_{\rm per}([0,1],\mathbb{C})\times L^{2}_{\rm per}([0,1],\mathbb{C})
\to L^{2}_{\rm per}([0,1],\mathbb{C})\times L^{2}_{\rm per}([0,1],\mathbb{C})
\end{eqnarray*}
with the domain $\mathcal{D}(\mathcal{F}(\epsilon))=H^{2}_{\rm per}([0,1],\mathbb{C})\times H^{1}_{\rm per}([0,1],\mathbb{C})$
for each $\epsilon\in (0,\epsilon_{0}]$.
For each $(u,w)^{T}$ in $L^{2}_{\rm per}([0,1],\mathbb{C})\times L^{2}_{\rm per}([0,1],\mathbb{C})$,
let the norm of $(u,w)^{T}$ be defined by
\begin{eqnarray*}
|(u,w)^{T}|_{2}=|u|_{2}+|w|_{2},
\end{eqnarray*}
where $|u|_{2}$ and $|w|_{2}$ are the $L^{2}$-norms of $u$ and $w$ in $L^{2}([0,1],\mathbb{C})$, respectively.

On the spectra $\sigma(\mathcal{L}_{0})$ and $\sigma(\mathcal{F}(\epsilon))$,  we have the next lemma.
\begin{lemma} \label{lm-spect-1}
The spectra $\sigma(\mathcal{L}_{0})$ and $\sigma(\mathcal{F}(\epsilon))$  for each $\epsilon\in(0,\epsilon_{0}]$
are both composed by a countable number of eigenvalues,
which have finite algebraic multiplicities.
Furthermore, if $\lambda$ is an eigenvalue of $\mathcal{F}(\epsilon)$,
then $\lambda/A_{\epsilon}^{2}$ is  an eigenvalue of  $\mathcal{L}_{0}$.
\end{lemma}
\begin{proof}
The first statement can be obtained by \cite[Lemma 8.4.1, p. 242]{Kapitula-Promislow-13}
and the second one can be obtained by a direct computation.
This finishes the proof.
\end{proof}

In order to analyze the spectrum $\sigma(\mathcal{F}(\epsilon))$
by the perturbation theory for linear operators,
we  decompose $\mathcal{F}(\epsilon)$ into two parts:
\begin{eqnarray*}
\mathcal{F}(\epsilon)=\mathcal{F}_{0}+\mathcal{F}_{1}(\epsilon),
\end{eqnarray*}
where
\begin{eqnarray} \label{F-0}
\mathcal{F}_{0}
\left(
\begin{array}{c}
u \\
w
\end{array}
\right)
:=\left(
\begin{array}{c}
u_{\xi\xi}-c_{0}Au_{\xi}+A^{2}f_{u}(P_{0},\alpha_{0})u+A^{2}f_{w}(P_{0},\alpha_{0})w \\
-c_{0}Aw_{\xi}+A^{2}g_{u}(P_{0},\alpha_{0})u+A^{2}g_{w}(P_{0},\alpha_{0})w
\end{array}
\right),
\end{eqnarray}
and $\mathcal{F}_{1}(\epsilon)=\mathcal{F}(\epsilon)-\mathcal{F}_{0}$, that is,
\begin{eqnarray} \label{F-1}
\mathcal{F}_{1}(\epsilon)
\left(
\begin{array}{c}
u \\
w
\end{array}
\right)
:= \left(
\begin{array}{c}
(c_{0}A-c A_{\epsilon})u_{\xi}+\beta_{u} u+\beta_{w}w \\
(c_{0}A-c A_{\epsilon})w_{\xi}+\gamma_{u}u+\gamma_{w}w
\end{array}
\right).
\end{eqnarray}
Recall that for each $\epsilon\in (0,\epsilon_{0}]$,
\begin{eqnarray*}
\alpha_{1}(\epsilon)=O(\epsilon),\ c_{1}(\epsilon)=O(\epsilon),
\ |(\phi,\phi_{\xi},\psi)|=O(\epsilon),\ |A_{\epsilon}-A|=O(\epsilon),
\end{eqnarray*}
then  we have that
$c_{0}A-c A_{\epsilon}=O(\epsilon)$,
\begin{eqnarray} \label{est-eps}
\begin{aligned}
\beta_{u}(\xi,\epsilon) &= A_{\epsilon}^{2}f_{u}(\phi,\psi,\alpha)-A^{2}f_{u}(P_{0},\alpha_{0})\!\!\!&=O(\epsilon),\\
\beta_{w}(\xi,\epsilon)  &= A_{\epsilon}^{2}f_{w}(\phi,\psi,\alpha)-A^{2}f_{w}(P_{0},\alpha_{0})\!\!\!&=O(\epsilon),\\
\gamma_{u}(\xi,\epsilon)&= A_{\epsilon}^{2}g_{u}(\phi,\psi,\alpha)-A^{2}g_{u}(P_{0},\alpha_{0})\!\!\!&=O(\epsilon),\\
\gamma_{w}(\xi,\epsilon) &= A_{\epsilon}^{2}g_{w}(\phi,\psi,\alpha)-A^{2}g_{w}(P_{0},\alpha_{0})\!\!\!&=O(\epsilon),
\end{aligned}
\end{eqnarray}
for $\epsilon\in (0,\epsilon_{0}]$.
As a consequence,  the operators  $\mathcal{F}(\epsilon)$ for $\epsilon\in(0,\epsilon]$
can be viewed as the perturbations of $\mathcal{F}_{0}$.
Next we give a further study on the relation between
$\mathcal{F}_{0}$ and $\mathcal{F}_{1}(\epsilon)$.
\begin{lemma} \label{lm-relat-bdd}
Let the operators $\mathcal{F}_{0}$ and $\mathcal{F}_{1}(\epsilon)$ for $\epsilon\in(0,\epsilon_{0}]$
be defined  by \eqref{F-0} and \eqref{F-1}, respectively.
Then for each $\epsilon\in (0,\epsilon_{0}]$ and each $\lambda\in\mathbb{C}$,
the operator $\mathcal{F}_{1}(\epsilon)$ is relatively bounded with respect to
$\mathcal{F}_{0}-\lambda I$,
that is, $\mathcal{D}(\mathcal{F}_{0}-\lambda I)\subset\mathcal{D}(\mathcal{F}_{1}(\epsilon))$
and there exists two nonnegative constants $a(\epsilon)$ and $b(\epsilon)$ such that
\begin{eqnarray} \label{ineq-relat-bdd}
|\mathcal{F}_{1}(\epsilon)(u,w)^{T}|_{2}\leq a(\epsilon) |(u,w)^{T}|_{2}+b(\epsilon) |(\mathcal{F}_{0}-\lambda I)(u,w)^{T}|_{2},
\end{eqnarray}
for each $(u,w)^{T}$ in $\mathcal{D}(\mathcal{F}_{0})$,
where the constants $a(\epsilon)$ and $b(\epsilon)$ can be in the form
\begin{eqnarray*}
a(\epsilon)\!\!\!&=&\!\!\!
   \kappa_{1}(\epsilon)\left\{1+\frac{2\kappa_{2}(\lambda)+2n(n+1)}{n-1-|c_{0}A|}+\frac{2\kappa_{2}(\lambda)}{|c_{0}A|}\right\},\\
b(\epsilon) \!\!\!&=&\!\!\!
   \kappa_{1}(\epsilon)\left\{\frac{1}{n-1-|c_{0}A|}+\frac{1}{|c_{0}A|}\right\},
\end{eqnarray*}
where $n$  is an integer with $n>1+|c_{0}A|$, and $\kappa_{1}(\epsilon)$ and $\kappa_{2}(\lambda)$ are in the form
\begin{eqnarray*}
\kappa_{1}(\epsilon)\!\!\!&=&\!\!\! 3\cdot\max\{|c_{0}A-cA_{\epsilon}|,
|\beta_{u}(\cdot,\epsilon)|_{\infty}+|\gamma_{w}(\cdot,\epsilon)|_{\infty},
\,|\beta_{w}(\cdot,\epsilon)|_{\infty}+|\gamma_{u}(\cdot,\epsilon)|_{\infty}\},\\
\kappa_{2}(\lambda) \!\!\!&=&\!\!\! \max\{A^{2}(|f_{u}(P_{0},\alpha_{0})|+|g_{w}(P_{0},\alpha_{0})|),
     A^{2}(|f_{w}(P_{0},\alpha_{0})|+|g_{u}(P_{0},\alpha_{0})|)\}+\frac{|\lambda|}{2}.
\end{eqnarray*}
Furthermore, the following limits hold:
\begin{eqnarray} \label{lmt}
a(\epsilon)\to 0,\ \ \ b(\epsilon)\to 0, \ \ \mbox{ as }\ \epsilon\to 0.
\end{eqnarray}
\end{lemma}
\begin{proof}
Noting that
\begin{eqnarray*}
\mathcal{F}_{0}:
L^{2}_{\rm per}([0,1],\mathbb{C})\times L^{2}_{\rm per}([0,1],\mathbb{C})
\to L^{2}_{\rm per}([0,1],\mathbb{C})\times L^{2}_{\rm per}([0,1],\mathbb{C})
\end{eqnarray*}
has the domain $\mathcal{D}(\mathcal{F}_{0})=H^{2}_{\rm per}([0,1],\mathbb{C})\times H^{1}_{\rm per}([0,1],\mathbb{C})$,
and
\begin{eqnarray*}
\mathcal{F}_{1}(\epsilon):
L^{2}_{\rm per}([0,1],\mathbb{C})\times L^{2}_{\rm per}([0,1],\mathbb{C})
\to L^{2}_{\rm per}([0,1],\mathbb{C})\times L^{2}_{\rm per}([0,1],\mathbb{C})
\end{eqnarray*}
has the domain $\mathcal{D}(\mathcal{F}_{1}(\epsilon))=H^{1}_{\rm per}([0,1],\mathbb{C})\times H^{1}_{\rm per}([0,1],\mathbb{C})$,
we then conclude that $\mathcal{D}(\mathcal{F}_{0})\subset\mathcal{D}(\mathcal{F}_{1}(\epsilon))$
for each $\epsilon\in(0,\epsilon_{0}]$.
This finishes the proof for the first statement.

Next we prove \eqref{ineq-relat-bdd}.
For each $\epsilon\in(0,\epsilon_{0}]$ and $(u,w)^{T}$ in $\mathcal{D}(\mathcal{F}_{0})$,
\begin{eqnarray} \label{est-F-1}
\begin{aligned}[rcl]
|\mathcal{F}_{1}(\epsilon)(u,w)^{T}|_{2}
&=|((c_{0}A-cA_{\epsilon})u_{\xi}+\beta_{u} u+\beta_{w}w,(c_{0}A-cA_{\epsilon})w_{\xi}+\gamma_{u}u+\gamma_{w}w)^{T}|_{2}\\
&\leq \kappa_{1}(\epsilon)\left(|u_{\xi}|_{2}+|w_{\xi}|_{2}+|(u,w)^{T}|_{2}\right),
\end{aligned}
\end{eqnarray}
where  $\kappa_{1}$ is defined as in this lemma.
By  \eqref{F-0}, we have
\begin{eqnarray*}
&&\lefteqn{|(u_{\xi\xi}-c_{0}Au_{\xi},-c_{0}Aw_{\xi})^{T}|_{2}}\\
&&= |\mathcal{F}_{0}(u,w)^{T}-(A^{2}f_{u}(P_{0},\alpha_{0})u
+A^{2}f_{w}(P_{0},\alpha_{0})w,A^{2}g_{u}(P_{0},\alpha_{0})u+A^{2}g_{w}(P_{0},\alpha_{0})w)^{T}|_{2}.
\end{eqnarray*}
Then
\begin{eqnarray*}
|c_{0}Aw_{\xi}|_{2} \!\!\!&\leq&\!\!\! |(\mathcal{F}_{0}-\lambda I)(u,w)^{T}|_{2}+2\kappa_{2}(\lambda)|(u,w)^{T}|_{2}, \\
|u_{\xi\xi}-c_{0}Au_{\xi}|_{2} \!\!\!&\leq&\!\!\!  |(\mathcal{F}_{0}-\lambda I)(u,w)^{T}|_{2}+2\kappa_{2}(\lambda)|(u,w)^{T}|_{2},
\end{eqnarray*}
where $\kappa_{2}(\lambda)$ is defined as in this lemma.
These yield that
\begin{eqnarray}
|w_{\xi}|_{2} \!\!\!&\leq&\!\!\!
|c_{0}A|^{-1}\left(|(\mathcal{F}_{0}-\lambda I)(u,w)^{T}|_{2}+2\kappa_{2}(\lambda)|(u,w)^{T}|_{2}\right), \label{est-w-1d}\\
|u_{\xi\xi}|_{2} \!\!\!&\leq&\!\!\!
|(\mathcal{F}_{0}-\lambda I)(u,w)^{T}|_{2}+2\kappa_{2}(\lambda)|(u,w)^{T}|_{2}+|c_{0}Au_{\xi}|_{2}. \label{est-u}
\end{eqnarray}
By the Inequality (1.12) in \cite[p.192]{Kato-95}, we have that for each $n>1$,
\begin{eqnarray*}
|u_{\xi}|_{2}\leq \frac{1}{n-1}|u_{\xi\xi}|_{2}+\frac{2n(n+1)}{n-1}|u|_{2}.
\end{eqnarray*}
This together with \eqref{est-u} yields that for $n>1$,
\begin{eqnarray*}
\begin{aligned}
\left(1-\frac{|c_{0}A|}{n-1}\right)|u_{\xi}|_{2}
\leq&\ \frac{1}{n-1}\left(|(\mathcal{F}_{0}-\lambda I)(u,w)^{T}|_{2}+2\kappa_{2}(\lambda)|(u,w)^{T}|_{2}\right)\\
&+\frac{2n(n+1)}{n-1}|u|_{2}.
\end{aligned}
\end{eqnarray*}
Hence, for every $n>1+|c_{0}A|$, we have
\begin{eqnarray} \label{est-u-1d}
\begin{aligned}
|u_{\xi}|_{2}
\leq&\ \frac{1}{n-1-|c_{0}A|}\left(|(\mathcal{F}_{0}-\lambda I)(u,w)^{T}|_{2}+2\kappa_{2}(\lambda)|(u,w)^{T}|_{2}\right)\\
&+\frac{2n(n+1)}{n-1-|c_{0}A|}|u|_{2}.
\end{aligned}
\end{eqnarray}
Substituting \eqref{est-w-1d} and \eqref{est-u-1d} into \eqref{est-F-1} yields that
\eqref{ineq-relat-bdd} holds.

Finally, we prove \eqref{lmt}.
By \eqref{est-eps} we have $\kappa_{1}(\epsilon)=O(\epsilon)$ for $\epsilon\in(0,\epsilon_{0}]$.
Then by the definitions of $a(\epsilon)$ and $b(\epsilon)$, we obtain \eqref{lmt}.
This finishes the proof.
\end{proof}

Letting $\lambda=0$ in  Lemma \ref{lm-relat-bdd}, we have the next corollary.
\begin{corollary} \label{cor-bdd}
For each $\epsilon\in (0,\epsilon_{0}]$,
the operator $\mathcal{F}_{1}(\epsilon)$ is relatively bounded with respect to
$\mathcal{F}_{0}$, and
\begin{eqnarray*}
|\mathcal{F}_{1}(\epsilon)(u,w)^{T}|_{2}\leq a(\epsilon) |(u,w)^{T}|_{2}+b(\epsilon) |\mathcal{F}_{0}(u,w)^{T}|_{2},
\end{eqnarray*}
for each $(u,w)^{T}$ in $\mathcal{D}(\mathcal{F}_{0})$,
where the constants $a(\epsilon)$ and $b(\epsilon)$ are defined as in Lemma \ref{lm-relat-bdd}.
\end{corollary}

We remark that in order to study the spectral stability of periodic traveling waves arising from Hopf bifurcation
in a class of systems of scalar viscous balance laws in one space dimension,
the recent work \cite{Alvarez-Plaza-21} also obtained the similar result as Corollary \ref{cor-bdd}.
We next study  the spectral perturbation problem for $\mathcal{F}_{0}$.
By \cite[Proposition 15.3, p. 151]{Hislop-Sigal} we have the following lemma.

\begin{lemma} \label{lm-pert}
Assume that $\lambda=\lambda_{0}$ is an isolated point spectrum of $\mathcal{F}_{0}$
and has finite algebraic multiplicity $m(\lambda_{0})$.
Let $\Gamma\subset \rho(\mathcal{F}_{0})$  be a small cycle surrounding $\lambda_{0}$
such that $\lambda_{0}$ is the only spectrum point inside the bounded open set surrounded by $\Gamma$.
Assume that the following statements hold:
\begin{enumerate}
\item[{\bf (i)}]
there exists a small constant $\tilde{\epsilon}_{0}>0$ such that
$\Gamma\subset \rho(\mathcal{F}(\epsilon))$ for all $\epsilon$ in $(0,\tilde{\epsilon}_{0}]$.
\item[{\bf (ii)}]
$|P_{\Gamma}(\mathcal{F}(\epsilon))-P_{\Gamma}(\mathcal{F}_{0})| \to 0$ as $\epsilon\to 0$.
\end{enumerate}
Then for sufficiently small $\epsilon$,
there are $\tilde{m}_{\epsilon}$ eigenvalues $\lambda_{1}(\epsilon),...,\lambda_{\tilde{m}_{\epsilon}}(\epsilon)$
 of $\mathcal{F}(\epsilon)$ inside $\Gamma$
and $\lambda_{j}(\epsilon)\to \lambda_{0}$ as $\epsilon\to 0$,
 and the total algebraic multiplicity of  $\lambda_{j}(\epsilon)$ is equal to $m(\lambda_{0})$.
\end{lemma}

By applying Lemmas \ref{lm-relat-bdd} and \ref{lm-pert}, we have the following results on the spectral perturbation.

\begin{lemma} \label{lm-spect-pert}
Assume that $\lambda=\lambda_{0}$ is an isolated point spectrum of $\mathcal{F}_{0}$
and has finite algebraic multiplicity $m(\lambda_{0})$.
Let $\Gamma\subset \rho(\mathcal{F}_{0})$  be a small cycle surrounding $\lambda_{0}$
such that $\lambda_{0}$ is the only spectrum point inside the bounded open set surrounded by $\Gamma$.
Then there exists a constant $\tilde{\epsilon}_{0}$ with $0<\tilde{\epsilon}_{0}\leq \epsilon_{0}$
such that for each $\epsilon$ in $(0,\tilde{\epsilon}_{0}]$,
there are $\tilde{m}_{\epsilon}$ eigenvalues $\lambda_{1}(\epsilon),...,\lambda_{\tilde{m}_{\epsilon}}(\epsilon)$
 of $\mathcal{F}(\epsilon)$ inside $\Gamma$
and $\lambda_{j}(\epsilon)\to \lambda_{0}$ as $\epsilon\to 0$,
 and the total algebraic multiplicity of  $\lambda_{j}(\epsilon)$ is equal to $m(\lambda_{0})$.
\end{lemma}
\begin{proof}
In order to prove this lemma, it suffices to check the conditions {\bf (i)} and {\bf (ii)} in Lemma \ref{lm-pert} hold.
Since $(\lambda I-\mathcal{F}_{0})^{-1}$ is analytic for $\lambda\in\rho(\mathcal{F}_{0})$,
there exists a constant $M(\Gamma)>0$, which is dependent of $\Gamma$, such that
\begin{eqnarray} \label{ineq-res-1}
|(\lambda I-\mathcal{F}_{0})^{-1}|\leq M(\Gamma), \ \ \ \ \ \lambda\in \Gamma.
\end{eqnarray}
Recall that $a(\epsilon)$ and $b(\epsilon)$ are defined as in Lemma \ref{lm-relat-bdd}
and satisfy the limits in \eqref{lmt}.
Then there exists a constant $\tilde{\epsilon}_{0}$ with $0<\tilde{\epsilon}_{0}\leq \epsilon_{0}$
such that
\begin{eqnarray*}
a(\epsilon)|(\lambda-\mathcal{F}_{0})^{-1}|+b(\epsilon)<1,\ \ \ \ \ \lambda\in \Gamma, \ \ \epsilon\in(0,\tilde{\epsilon}_{0}].
\end{eqnarray*}
Hence by Lemma  \ref{lm-relat-bdd} and \cite[Theorem 1.16, p.\,196]{Kato-95},
we have $\Gamma\subset \rho(\mathcal{F}(\epsilon))$ for all $\epsilon$ in $(0,\tilde{\epsilon}_{0}]$, and
\begin{eqnarray}
|(\lambda I-\mathcal{F}(\epsilon))^{-1}|
\!\!\!&\leq&\!\!\!   \frac{|(\lambda I-\mathcal{F}_{0})^{-1}|}{1-a(\epsilon)|(\lambda I-\mathcal{F}_{0})^{-1}|-b(\epsilon)},\nonumber\\
|(\lambda I-\mathcal{F}(\epsilon))^{-1}-(\lambda I-\mathcal{F}_{0})^{-1}|
\!\!\!&\leq&\!\!\!\frac{|(\lambda I-\mathcal{F}_{0})^{-1}|(a(\epsilon)|(\lambda I-\mathcal{F}_{0})^{-1}|+b(\epsilon))}
{1-a(\epsilon)|(\lambda I-\mathcal{F}_{0})^{-1}|-b(\epsilon)},\label{ineq-res-2}
\end{eqnarray}
for $\lambda\in \Gamma$ and $\epsilon\in(0,\tilde{\epsilon}_{0}]$.
This yields that $\Gamma\subset \rho(\mathcal{F}(\epsilon))$ for all $\epsilon$ in $(0,\tilde{\epsilon}_{0}]$.

Let the spectral projections $P_{\Gamma}(\mathcal{F}_{0})$
and $P_{\Gamma}(\mathcal{F}(\epsilon))$ for $\epsilon\in(0,\tilde{\epsilon}_{0}]$  be defined by
\begin{eqnarray*}
P_{\Gamma}(\mathcal{F}_{0})\!\!\!&=&\!\!\!\frac{1}{2\pi {\bf i}}\oint_{\Gamma}(\lambda I-\mathcal{F}_{0})^{-1}\, d\lambda,\\
P_{\Gamma}(\mathcal{F}(\epsilon))\!\!\!&=&\!\!\!
\frac{1}{2\pi {\bf i}}\oint_{\Gamma}(\lambda I-\mathcal{F}(\epsilon))^{-1}\, d\lambda,
\end{eqnarray*}
respectively. Then by \eqref{ineq-res-1} and \eqref{ineq-res-2}, we have
\begin{eqnarray*}
|P_{\Gamma}(\mathcal{F}_{0})-P_{\Gamma}(\mathcal{F}(\epsilon))|
\leq \frac{1}{2\pi }\oint_{\Gamma}|(\lambda I-\mathcal{F}(\epsilon))^{-1}-(\lambda I-\mathcal{F}_{0})^{-1}|\, |d\lambda|.
\end{eqnarray*}
Consequently,
by the limits in \eqref{lmt}, the inequalities \eqref{ineq-res-1} and \eqref{ineq-res-2}, we have that
\begin{eqnarray*}
|P_{\Gamma}(\mathcal{F}(\epsilon))-P_{\Gamma}(\mathcal{F}_{0})| \to 0, \ \ \ \mbox{ as }\ \ \epsilon\to 0.
\end{eqnarray*}
Thus, the proof for this lemma is finished by applying Lemma \ref{lm-pert}.
\end{proof}

\begin{remark}
The perturbation theory for a family of linear operators
$H_{\kappa}=H+\kappa W$ was considered in Chapter 15 of \cite{Hislop-Sigal},
where $H_{\kappa}$ are required to be analytic in $\kappa$.
In Lemma \ref{lm-spect-pert}, we do not require the smoothness of $\mathcal{F}(\epsilon)$  in $\epsilon$.
From Lemma \ref{lm-spect-pert},  the point spectrum of $\mathcal{F}_{0}$ is often said to be {\it stable
under the perturbations $\mathcal{F}_{1}(\epsilon)$} for sufficiently small $\epsilon>0$.
\end{remark}

\subsection{Instability of periodic traveling waves}
\label{pf-thm-stab}

In order to study the stability of periodic traveling waves emerging from fold-Hopf equilibria,
we first study the spectrum $\sigma(\mathcal{F}_{0})$ of $\mathcal{F}_{0}$.

\begin{lemma} \label{lm-orig-operator}
Let the operator $\mathcal{F}_{0}$ be defined by \eqref{F-0}.
Assume that $A>0$.
Then the following statements hold:
\begin{enumerate}

\item[{\bf (i)}]
$\mathcal{F}_{0}$ has an positive eigenvalue $\lambda=A^{2}\mu_{0}^{2}$
with an eigenfunction $(f_{w}(P_{0},\alpha_{0}),g_{w}(P_{0},\alpha_{0}))^{T}$.

\item[\bf (ii)]
the algebraic multiplicity and the geometric multiplicity of $\lambda=A^{2}\mu_{0}^{2}$
are both equal to one.
\end{enumerate}
\end{lemma}
\begin{proof}
For simplification, we only give the proof for  $A=1$.
For a general  $A>0$, the results in this lemma can be obtained by taking a rescaling $\xi\to \xi/A$ in  $\mathcal{F}_{0}(u,w)^{T}=A^{2}\lambda(u,w)^{T}$.
Set $v=u_{\xi}$ and $Y=(u,v,w)^{T}$.
Then the eigenvalue problem  $\mathcal{F}_{0}(u,w)^{T}=\lambda(u,w)^{T}$ is equivalent to
the following system
\begin{eqnarray} \label{eq-spect-0}
\frac{d Y}{d \xi}=A_{0}(\lambda)Y, \ \ \
A_{0}(\lambda)=\left(
\begin{array}{ccc}
0 & 1 & 0\\
\lambda-f_{u}(P_{0},\alpha_{0}) & c_{0} & -f_{w}(P_{0},\alpha_{0})\\
\frac{1}{c_{0}}g_{u}(P_{0},\alpha_{0}) & 0 & \frac{1}{c_{0}}(g_{w}(P_{0},\alpha_{0})-\lambda)
\end{array}
\right)
\end{eqnarray}
subject to the periodic boundary condition $Y(1)=Y(0)$.
By Lemma \ref{lm-FH-pt} we have
\begin{eqnarray*}
{\rm det}(A_{0}(\lambda))\!\!\!&=&\!\!\!-\frac{1}{c_{0}}\left(\lambda^{2}
-(f_{u}(P_{0},\alpha_{0})+g_{w}(P_{0},\alpha_{0})\lambda\right.\\
\!\!\!& &\!\!\!\left.+f_{u}(P_{0},\alpha_{0})g_{w}(P_{0},\alpha_{0})
-f_{w}(P_{0},\alpha_{0})g_{u}(P_{0},\alpha_{0})\right)\\
\!\!\!&=&\!\!\!-\frac{1}{c_{0}}\lambda\left(\lambda
-\mu_{0}^{2}\right).
\end{eqnarray*}
Then for $\lambda=\mu_{0}^{2}>0$, system \eqref{eq-spect-0} has a nonzero constant solution
$(f_{w}(P_{0},\alpha_{0}),0,g_{w}(P_{0},\alpha_{0}))^{T}$.
This proves {\bf (i)}.

By  \eqref{cond-2} and substituting $\lambda=\mu_{0}^{2}$ into \eqref{eq-spect-0},
we have that the rank of $A_{0}(\mu_{0}^{2})$ is two. Then
${\rm Ker}(\mathcal{F}_{0}-\mu_{0}^{2} I)$ is one-dimensional and spanned by it eigenfunction $(f_{w}(P_{0},\alpha_{0}),0,g_{w}(P_{0},\alpha_{0}))^{T}$.
Hence, the geometric multiplicity of $\lambda=\mu_{0}^{2}$ is one.
Since  the characteristic polynomial associated with $A_{0}(\lambda)$ is in the form
\begin{eqnarray*}
G(z,\lambda)\!\!\!&:=&\!\!\! z\left(z^{2}-\left(c_{0}+\frac{1}{c_{0}}(g_{w}(P_{0},\alpha_{0})-\lambda)\right)z
             +f_{u}(P_{0},\alpha_{0})+g_{w}(P_{0},\alpha_{0})-2\lambda\right)\\
          \!\!\!& &\!\!\!
          \!\!\!-\frac{1}{c_{0}}\!\left(\lambda^{2}-(f_{u}(P_{0},\alpha_{0})+g_{w}(P_{0},\alpha_{0}))\lambda+\!
          f_{u}(P_{0},\alpha_{0})g_{w}(P_{0},\alpha_{0})-\!f_{w}(P_{0},\alpha_{0})g_{u}(P_{0},\alpha_{0})\right),
\end{eqnarray*}
 by \eqref{cond-2} we have
\begin{eqnarray*}
G(z,\lambda)= z\left(z^{2}+\frac{1}{c_{0}}\lambda z
             +\mu_{0}^{2}-2\lambda\right)
          -\frac{1}{c_{0}}\left(\lambda^{2}-\mu_{0}^{2}\lambda\right).
\end{eqnarray*}
Let $z_{j}(\lambda)$, $j=1,2,3$, denote the zeros of $G(z,\lambda)$ for each $\lambda\in\mathbb{C}$.
Especially, for $\lambda=\mu_{0}^{2}$ the zeros of $G(z,\mu_{0}^{2})$ are
\begin{eqnarray*}
z_{1}(\mu_{0}^{2})=0,\ \ \ \ \ z_{2,3}(\mu_{0}^{2})=-\frac{1}{2c_{0}}\mu_{0}^{2}\pm\sqrt{\mu_{0}^{2}+\frac{1}{4c_{0}^{2}}\mu_{0}^{4}}\neq 0,
\end{eqnarray*}
and $z'_{1}(\mu_{0}^{2})=-\frac{1}{c_{0}}$.
Then the Evans function $E(\lambda,\mu)$ in \eqref{df-evans} satisfies
\begin{eqnarray*}
E(\lambda,0)={\rm det}(e^{A_{0}(\lambda)}-I_{3})=\prod_{j=1}^{3}(e^{z_{j}(\lambda)}-1).
\end{eqnarray*}
Hence, $E(\mu_{0}^{2},0)=0$ and
\begin{eqnarray*}
\frac{\partial E}{\partial \lambda}(\mu_{0}^{2},0)
=z'_{1}(\mu_{0}^{2})\cdot e^{z_{1}(\mu_{0}^{2})}\cdot(e^{z_{2}(\mu_{0}^{2})}-1)\cdot(e^{z_{3}(\mu_{0}^{2})}-1)\neq 0.
\end{eqnarray*}
This together with \cite[Lemma 8.4.1]{Kapitula-Promislow-13} yields {\bf (ii)}.
Thus, the proof is now complete.
\end{proof}

By Lemmas \ref{lm-spect-pert} and \ref{lm-orig-operator},
we have the following theorem.

\begin{theorem}  \label{thm-stab-PTW} {\rm (Instability of periodic traveling waves)}
Assume that there exists  a sufficiently small $\epsilon_{0}>0$,
and two continuous functions  $\alpha_{1}(\epsilon)$ and $c_{1}(\epsilon)$
with $\alpha_{1}(\epsilon)=O(\epsilon)$ and $c_{1}(\epsilon)=O(\epsilon)$
for $\epsilon\in (0,\epsilon_{0}]$
such that a small-amplitude periodic solution $(\phi,\phi_{\xi},\psi)$ with period $A_{\epsilon}$
bifurcates from a fold-Hopf point $P_{0}$
in the 3D system \eqref{3D-SYSTEM} with $(\alpha,c)=(\alpha_{0}+\alpha_{1}(\epsilon),c_{0}+c_{1}(\epsilon))$
for each $\epsilon\in(0,\epsilon_{0}]$,
where
\begin{eqnarray*}
|(\phi,\phi_{\xi},\psi)-P_{0}|=O(\epsilon),\ \ \ \ A_{\epsilon}=A+O(\epsilon), \ \ \ A>0.
\end{eqnarray*}
Then there exists a constant $\tilde{\epsilon}_{0}$ with $0<\tilde{\epsilon}_{0}\leq \epsilon_{0}$
such that for each $\epsilon$ with $0<\epsilon\leq \tilde{\epsilon}_{0}$,
the corresponding periodic traveling wave $(\phi,\psi)$ with wave speed $c_{0}+c_{1}(\epsilon)$
in the system \eqref{PDE-ODE} with $\alpha_{0}+\alpha_{1}(\epsilon)$ are spectrally unstable.
\end{theorem}
\begin{proof}
By Lemma \ref{lm-orig-operator}, we have that
$\mathcal{F}_{0}$ has an positive eigenvalue $\lambda_{0}=A^{2}\mu_{0}^{2}$,
whose algebraic multiplicity is one.
Let $\Gamma$ in ${\rm Re}\lambda>0$  be a small cycle surrounding $A^{2}\mu_{0}^{2}$
such that $\lambda_{0}=A^{2}\mu_{0}^{2}$ is the only spectrum point inside $\Gamma$.
By Lemma \ref{lm-spect-pert},
there exists a constant $\tilde{\epsilon}_{0}$ with $0<\tilde{\epsilon}_{0}\leq \epsilon_{0}$
such that for each $\epsilon$ in $(0,\tilde{\epsilon}_{0}]$,
there exists a unique eigenvalue $\lambda_{1}(\epsilon)$ of $\mathcal{F}(\epsilon)$ inside $\Gamma$
and $\lambda_{1}(\epsilon)\to \lambda_{0}$ as $\epsilon\to 0$.
Then $\sigma(\mathcal{F}(\epsilon))$ has an eigenvalue $\lambda_{1}(\epsilon)$ with positive real part
for $\epsilon\in(0,\tilde{\epsilon}_{0}]$.
This together with Lemma \ref{lm-spect-1}
yields that the periodic traveling wave $(\phi,\psi)$ is spectrally unstable.
Therefore, the proof is now complete.
\end{proof}

As a corollary of  Theorem \ref{thm-stab-PTW},
we have the stability of periodic traveling waves obtained in Theorem \ref{thm-existence}.

\begin{corollary}
Assume that $(\phi,\psi)$ is a periodic traveling wave obtained as in Theorem \ref{thm-existence}.
Then this periodic traveling wave  is spectrally unstable.
\end{corollary}

\section{Applications to the FitzHugh-Nagumo system}
\label{sec-app}

In this section,
we only apply the main results to the FitzHugh-Nagumo system.
By the similar method,
one can study the sufficient conditions under which small-amplitude periodic traveling waves
bifurcate from fold-Hopf equilibria in caricature calcium models \cite{Tsai-etal-2012,Zhang-Sneyd-Tsai-14},
consumer-resource models \cite{He-Lam-19,Zhang-etal-17} and so on,
and rigorously prove the spectral instability for these traveling wave solutions.

\subsection{Mathematical model}

We follow \cite{Champneys-etal07} and study  the FitzHugh-Nagumo system with an applied current
\begin{eqnarray} \label{FHN-PDE}
\begin{aligned}
\frac{\partial u}{\partial t} &=d u_{xx}+h(u)-w+p,\\
\frac{\partial w}{\partial t} &=\delta(u-\gamma w),
\end{aligned}
\end{eqnarray}
where $x$ is a one-dimensional spatial variable, $t$ is time,
the state $u$ represnets the plasma membrane electric potential with the diffusion rate $d$,
the state $w$ describes the combined inactivation effects of potassium and sodium ion channels,
$p$ is the applied current,  $\delta$ is
the ratio of the time scale of the membrane potential to that of the channels,
$h$ is a bistable function in the form
$$h(u)=u(u-1)(\alpha-u),$$
and all system parameters  are positive.

In the moving coordinate frame $(\xi,t)=(x+ct,t)$,
a periodic wave profile $(\phi(\cdot),\psi(\cdot))$ is a periodic steady state of the following system
\begin{eqnarray} \label{eq-TW-1}
\begin{aligned}
u_{t}&=d u_{\xi\xi}-cu_{\xi}+h(u)-w+p,\\
w_{t}&=-cw_{\xi}+\delta (u-\gamma w).
\end{aligned}
\end{eqnarray}
Set $v=u_{\xi}$. Then the vector $(\phi,\phi_{\xi},\psi)$ is a periodic solution of
the three-dimensional ordinary differential system
\begin{eqnarray} \label{3D-FHN}
\begin{aligned}
\frac{d u}{d \xi} &= \dot u  = v,
\\
\frac{d v}{d \xi} &= \dot v  =\frac{1}{d}\left( -h(u)+cv+w-p\right),
\\
\frac{d w}{d \xi} &= \dot w  = \frac{\delta}{c} (u-\gamma w).
\end{aligned}
\end{eqnarray}
For convenience, we call \eqref{3D-FHN} the {\it 3D-FHN system}.
Throughout this paper,
the values of the parameters $\alpha$, $d$ and $\delta$ are fixed as they were given in
\cite{Champneys-etal07,Czechowski-Piotr-16,Guck-Kuehn-09}, that is,
\begin{eqnarray} \label{val-param}
\alpha=0.1,\ \ \ d=5.0,\ \ \  \delta=0.01.
%,\ \ \ \gamma=1.0.
\end{eqnarray}
Then the 3D-FHN system \eqref{3D-FHN} depends on three parameters $c$, $\gamma$ and $p$.

\subsection{Fold-Hopf equilibrium}

If $(u_{0},v_{0},w_{0})^{T}$ in $\mathbb{R}^{3}$ is an equilibrium of the 3D-FHN system \eqref{3D-FHN},
then $(u_{0},v_{0},w_{0})^{T}$ satisfies the equations:
\begin{eqnarray*}
v_{0}=0,\ \ \ h(u_{0})+p=w_{0},\ \ \ u_{0}=\gamma w_{0}.
\end{eqnarray*}
This yields that  the equilibria of the 3D-FHN system \eqref{3D-FHN} are determined by the roots of equation
$$g(u,\gamma)=p,\ \ \ \ \  u\in\mathbb{R},\ \ \gamma,\ p>0,$$
where
\begin{eqnarray} \label{eq-imp}
g(u,\gamma):=\frac{1}{\gamma}u-h(u)=u^3-1.1 u^2+\left(0.1+\frac{1}{\gamma}\right)u.
\end{eqnarray}
For each $\gamma,p>0$,
the results on the zeros of $g(u,\lambda)=p$ are summarized as follows.
\begin{lemma} \label{lm-zero}
Let $N(\gamma,p)$ denote the number of the zeros for $g(u,\gamma)=p$ with $\gamma>0$ and $p>0$.
Then the following statements hold:
\begin{itemize}
\item[{\bf (i)}]
If $\gamma$ satisfies $0<\gamma\leq \frac{300}{91}$,
then $N(\gamma,p)=1$ for every $p>0$.

\item[{\bf (ii)}]
If $\gamma$ satisfies $\gamma>\frac{300}{91}$,
then $g_{u}(u,\gamma)=0$ has precisely two positive roots $u_{1}(\gamma)$ and $u_{2}(\gamma)$
with $0<u_{1}(\gamma)<u_{2}(\gamma)$ for each $\gamma>\frac{300}{91}$.
Let $u_{0}(\gamma)$ and $u_{3}(\gamma)$ be defined by
\begin{eqnarray*}
g(u_{0}(\gamma))=g(u_{2}(\gamma)),& \ \ \ &u_{0}(\gamma)<u_{1}(\gamma),\\
g(u_{3}(\gamma))=g(u_{1}(\gamma)),& \ \ \ &u_{3}(\gamma)>u_{2}(\gamma).
\end{eqnarray*}
Moreover, the following dichotomies hold:
\begin{itemize}
\item[{\bf (ii.1)}]
If $\gamma$ satisfies $\frac{300}{91}<\gamma<\frac{400}{81}$,
then $N(\gamma,p)=1$  for $p>g(u_{1}(\gamma))$ and $0<p<g(u_{2}(\gamma))$,
$N(\gamma,p)=2$  for $p=g(u_{1}(\gamma))$ and $p=g(u_{2}(\gamma))$,
and $N(\gamma,p)=3$  for $g(u_{2}(\gamma))<p<g(u_{1}(\gamma))$.

\item[{\bf (ii.2)}]
If $\gamma$ satisfies $\gamma\geq\frac{400}{81}$,
then $N(\gamma,p)=1$  for $p>g(u_{1}(\gamma))$,
$N(\gamma,p)=2$  for  $p=g(u_{1}(\gamma))$,
and $N(\gamma,p)=3$  for $0<p<g(u_{1}(\gamma))$.
\end{itemize}
\end{itemize}
\end{lemma}
%
%The graphes of $g(u,\gamma)$ for different $(\gamma,p)$ are shown in Figure \ref{fg-dist}.
%
%\begin{figure}[!htbp]
%\centering
%\subfigure[]{
%\begin{minipage}[t]{0.22\linewidth}
%\centering
%\includegraphics[width=1.1in]{Fig-g-1}
%\caption{fig2}
%\end{minipage}
%\label{fg-g-1}
%}%
%\subfigure[]{
%\begin{minipage}[t]{0.22\linewidth}
%\centering
%\includegraphics[width=1.1in]{Fig-g-2}
%\caption{fig2}
%\end{minipage}
%\label{fg-g-2}
%}%
%\subfigure[]{
%\begin{minipage}[t]{0.22\linewidth}
%\centering
%\includegraphics[width=1.1in]{Fig-g-3}
%\caption{fig2}
%\end{minipage}
%\label{fg-g-3}
%}%
%\subfigure[]{
%\begin{minipage}[t]{0.22\linewidth}
%\centering
%\includegraphics[width=1.1in]{Fig-g-4}
%%\caption{fig2}
%\end{minipage}
%\label{fg-g-4}
%}%
%\centering
%\caption{
% The graphes of $g$ with respect to $u$ for different values of $\gamma$:
%\ref{fg-g-1} $\gamma\leq \frac{300}{91}$,
% \ref{fg-g-2} $\frac{300}{91}<\gamma<\frac{400}{81}$,
% \ref{fg-g-3} $\gamma=\frac{400}{81}$ and
% \ref{fg-g-4} $\gamma>\frac{400}{81}$.
%}
%\label{fg-dist}
%\end{figure}
%
\begin{proof}
Note that the first-order partial derivative of $g$ with respect to $u$ is in the form
\begin{eqnarray*}
g_{u}(u,\gamma)=3u^{2}-2.2u+0.1+\frac{1}{\gamma}.
\end{eqnarray*}
Thus, $g_{u}(u,\gamma)\geq 0$ for all $u\geq 0$ and $0<\gamma\leq \frac{300}{91}$.
This yields {\bf (i)}.

For each $\gamma>\frac{300}{91}$, equation $g_{u}(u,\gamma)=0$ has precisely two zeros
$u_{1}(\gamma)$ and $u_{2}(\gamma)$ with $u_{1}(\gamma)<u_{2}(\gamma)$,
and $g(u,\gamma)$ has a zero of multiplicity two for $\gamma=\frac{400}{81}$.
Hence $u_{0}(\gamma)$ and $u_{3}(\gamma)$ are well-defined,
and the statements in {\bf (ii)} can be obtained by a standard analysis.
Thus, the proof is complete.
\end{proof}

We next study the conditions under which an equilibrium is a fold-Hopf equilibrium.
If the 3D-FHN system \eqref{3D-FHN} with \eqref{val-param} has
an equilibrium $(u_{0},v_{0},w_{0})^{T}$ in $\mathbb{R}^{3}$,
then $g(u_{0},\gamma)=p$.
Clearly, the Jacobian matrix $J(u_{0},v_{0},w_{0})$ at this point is in the form
\begin{equation*}
J(u_{0},v_{0},w_{0})=\left(
\begin{array}{ccc}
0 & 1 & 0\\
-\frac{1}{d}h'(u_{0}) & \frac{c}{d} & \frac{1}{d}\\
\frac{\delta}{c} & 0 & -\frac{\delta \gamma}{c}
\end{array}
\right),
\end{equation*}
and the eigenvalues $\lambda_{j}$, $j=1,2,3$,  of $J(u_{0},v_{0},w_{0})$ are determined by the zeros of its  characteristic polynomial
\begin{eqnarray*}
&&\lefteqn{{\rm det}(\lambda I_{3}-J(u_{0},v_{0},w_{0}))}\\
&&=\lambda^{3}+\left(\frac{\delta\gamma}{c}-\frac{c}{d}\right)\lambda^{2}
   +\frac{1}{d}\left(h'(u_{0})-\delta \gamma\right)\lambda
   +\frac{\delta}{cd}\left(\gamma h'(u_{0})-1)\right),
\end{eqnarray*}
where $I_{3}$ is the $3\times 3$ identity matrix.
Assume that $(u_{0},v_{0},w_{0})^{T}$ is a fold-Hopf equilibrium.
Then without loss of generality, we assume $\lambda_{1}=0$ and
$\lambda_{2,3}=\pm\, {\bf i}\mu_{0}$ for some $\mu_{0}>0$.
Then
\begin{eqnarray*}
\gamma h'(u_{0})-1=0, \ \ \
\sum_{j=1}^{3}\lambda_{j}=\frac{\delta\gamma}{c}-\frac{c}{d}=0,\ \ \
\frac{1}{d}(h'(u_{0})-\delta \gamma)=\mu_{0}^{2}>0.
\end{eqnarray*}
This yields that if  $(u_{0},v_{0},w_{0})^{T}$ in $\mathbb{R}^{3}$ is a fold-Hopf equilibrium,
then $(c,\gamma,p)$ satisfies
\begin{eqnarray} \label{cond-FH-point}
\begin{aligned}
&&g(u_{0},\gamma)=p,\ \ \ \frac{1}{\gamma}=h'(u_{0})>0.01\gamma>0, \\
&&c=\sqrt{0.05\gamma},
\ \ \ \mu_{0}^{2}=\frac{1}{500\gamma}(100-\gamma^{2})>0.
\end{aligned}
\end{eqnarray}
More properties on the fold-Hopf equilibria of the 3D-FHN system \eqref{3D-FHN} are summarized as follows.
\begin{lemma} \label{lm-FHP}
Let $c=\sqrt{0.05\gamma}$ for each $\gamma>0$.
Then the following statements hold:
\begin{itemize}
\item[{\bf (i)}] For $0<\gamma<\frac{300}{91}$ and $\gamma\geq 10$,
the 3D-FHN system \eqref{3D-FHN} has no fold-Hopf equilibria.

\item[{\bf (ii)}] For $\gamma=\frac{300}{91}$,
the 3D-FHN system \eqref{3D-FHN} has precisely one fold-Hopf equilibrium at $(\frac{11}{30},0, \frac{1001}{9000})^{T}$,
and $p$ takes the value $\frac{1331}{27000}$.

\item[{\bf (iii)}]
For $\frac{300}{91}<\gamma<\frac{400}{81}$,
the 3D-FHN system \eqref{3D-FHN} has precisely two fold-Hopf equilibria at
$(u_{1}(\gamma),0,u_{1}(\gamma)/\gamma)^{T}$ for $p=g(u_{1}(\gamma),\gamma)$
and $(u_{2}(\gamma),0,u_{2}(\gamma)/\gamma)^{T}$ for $p=g(u_{2}(\gamma),\gamma)$, respectively,
where $u_{1}(\gamma)$ and $u_{2}(\gamma)$ are defined as in  {\bf (ii)} of Lemma \ref{lm-zero}.

\item[\bf (iv)] For $\frac{400}{81}\leq \gamma<10$,
the 3D-FHN system \eqref{3D-FHN} has precisely one fold-Hopf equilibrium at $(u_{1}(\gamma),0,u_{1}(\gamma)/\gamma)^{T}$ for $p=g(u_{1}(\gamma),\gamma)$,
where $u_{1}(\gamma)$ is defined as in {\bf (ii)} of Lemma \ref{lm-zero}.
\end{itemize}
\end{lemma}
\begin{proof}
By the proof for {\bf (i)} of Lemma \ref{lm-zero},
we have that  $\frac{1}{\gamma}>h'(u)$ for each $u>0$ and $\gamma$ with $0<\gamma<\frac{300}{91}$.
Then {\bf (i)} holds for $0<\gamma<\frac{300}{91}$.
Since $\frac{1}{\gamma}\leq 0.01\gamma$ for $\gamma\geq 10$,
the proof for {\bf (i)} is finished by \eqref{cond-FH-point}.

Note that $\frac{1}{\gamma}=h'(u)$ implies  $g'(u,\gamma)=0$.
Then by Lemma \ref{lm-zero} and a direct computation, the remaining statements hold.
Thus, the proof is complete.
\end{proof}

\subsection{Existence and stability of periodic traveling waves}
\label{sec-exmp-thm}

 Before stating the main results in this section,
we remark that Tsai, Zhang, Kirk and Sneyd \cite{Tsai-etal-2012} gave some conditions under
which there exist periodic traveling waves arising from fold-Hopf equilibria in the FitzHugh-Nagumo system \eqref{FHN-PDE}.
However, they did not investigate the stability of these perturbed periodic traveling waves.

By Lemma \ref{lm-averag-1} and  Theorem \ref{thm-stab-PTW},
we can obtain the existence and stability of periodic traveling waves emerging fold-Hopf equilibria.
More precisely, we have  the  following theorem.

\begin{theorem} \label{thm-FHN-periodic}
Let $(\alpha, d,\delta)$ be in the form \eqref{val-param}.
Assume that the 3D-FHN system \eqref{3D-FHN} with $(c,\gamma,p)=(c_{0},\gamma_{0},p_{0})$
has a fold-Hopf equilibrium $(u_{0},v_{0},w_{0})^{T}$ whose linearization has the eigenvalues $\lambda_{1}=0$ and
$\lambda_{2,3}=\pm\, {\bf i}\mu_{0}$ for some $\mu_{0}>0$.
Then  there is a sufficiently small $\epsilon_{0}>0$,
and two continuous functions  $\gamma(\epsilon)$ and $p(\epsilon)$ given by
\begin{eqnarray*}
\gamma(\epsilon)=\gamma_{0}+\gamma_{1}\epsilon+O(\epsilon^{2}), \ \ \ \
p(\epsilon)=p_{0}+O(\epsilon^{3}),
\end{eqnarray*}
for $\epsilon\in (0,\epsilon_{0}]$ and $\gamma_{1}\neq 0$ with $\gamma_{1}(11-30u_{0})>0$,
such that a periodic solution $(\phi,\phi_{\xi},\psi)$ with period
$A_{\epsilon}$ bifurcates from $(u_{0},v_{0},w_{0})^{T}$
in the 3D-FHN system \eqref{3D-FHN}  with $(c,\gamma,p)=(c_{0}+O(\epsilon^{3}),\gamma(\epsilon),p(\epsilon))$
for each $\epsilon\in(0,\epsilon_{0}]$,
and
\begin{eqnarray*}
|(\phi,\phi_{\xi},\psi)-(u_{0},v_{0},w_{0})|=O(\epsilon),\ \ \ \ A_{\epsilon}=2\pi/\mu_{0}+O(\epsilon),
\ \ \ \epsilon\in(0,\epsilon_{0}].
\end{eqnarray*}
Furthermore, the corresponding periodic traveling wave $(\phi,\psi)$ in the FitzHugh-Nagumo system \eqref{FHN-PDE} is spectrally unstable.
\end{theorem}
\begin{proof}
We begin by proving the first statement by the averaging theory.
After translating  the equilibrium $(u_{0},v_{0},w_{0})$ of the 3D-FHN system \eqref{3D-FHN} to the origin,
we change the variables by
\begin{eqnarray} \label{chge-1}
(u,v,w)\to (\epsilon(u+\gamma_{0}w), \epsilon \mu_{0}v, \epsilon(\delta_{0}\gamma_{0}u-c_{0}\mu_{0}v+w)),
\end{eqnarray}
then the 3D-FHN system \eqref{3D-FHN} is converted into
\begin{eqnarray} \label{3D-FHN-1}
\begin{aligned}
\dot u  &= \mu_{0} v-\frac{1}{d_{0}\mu_{0}^{2}}(c_{0}g_{1}(u,w)+g_{2}(u,v,w))\epsilon+O(\epsilon^{2}),
\\
\dot v  &= -\mu_{0} u+\frac{1}{\mu_{0}}g_{1}(u,w)\epsilon+O(\epsilon^{2}),
\\
\dot w  &= \frac{1}{d_{0}\gamma_{0}\mu_{0}^{2}}(c_{0}g_{1}(u,w)+g_{2}(u,v,w))\epsilon+O(\epsilon^{2}),
\end{aligned}
\end{eqnarray}
where
\begin{eqnarray*}
g_{1}(u,w)\!\!\!&=&\!\!\!\frac{1}{d_{0}}(3u_{0}-1.1)(u+\gamma_{0}w)^{2},\\
g_{2}(u,v,w)\!\!\!&=&\!\!\!-\delta_{0}\gamma_{1}(\delta_{0}\gamma_{0}u-c_{0}\mu_{0}v+w).
\end{eqnarray*}
By another change of variables
\begin{eqnarray*}
(u,v,w)\to (r\cos\theta,r\sin\theta,w), \ \ \ \ \ r\geq 0,\ \ \ 2\pi \geq \theta\geq 0,
\end{eqnarray*}
we have
\begin{eqnarray} \label{3D-FHN-2}
\begin{aligned}
\dot r  &= \left(-\frac{1}{d_{0}\mu_{0}^{2}}(c_{0}g_{1}+g_{2})\cos\theta
      +\frac{1}{\mu_{0}}g_{1}\sin\theta\right)\epsilon+O(\epsilon^{2}),
\\
r\dot \theta  &= -\mu_{0} r+\left(\frac{1}{d_{0}\mu_{0}^{2}}(c_{0}g_{1}+g_{2})\sin\theta
      +\frac{1}{\mu_{0}}g_{1}\cos\theta\right)\epsilon+O(\epsilon^{2}),
\\
\dot w  &= \frac{1}{d_{0}\gamma_{0}\mu_{0}^{2}}(c_{0}g_{1}+g_{2})\epsilon+O(\epsilon^{2}),
\end{aligned}
\end{eqnarray}
where we write $g_{1}=g_{1}(r\cos\theta,w)$ and $g_{2}=g_{2}(r\cos\theta,r\sin\theta,w)$ for short.
Then for sufficiently small $|r|$ and $|\epsilon|$,
we change system (\ref{3D-FHN-1}) into  a nonautonomous system
\begin{eqnarray} \label{2D-FHN}
\begin{aligned}
\frac{d r}{d\theta} &=
      \left(\frac{1}{d_{0}\mu_{0}^{3}}(c_{0}g_{1}+g_{2})\cos\theta
      -\frac{1}{\mu_{0}^{2}}g_{1}\sin\theta\right)\epsilon+O(\epsilon^{2}):=H_{1}(\theta,r,w)\epsilon+O(\epsilon^{2}),
\\
\frac{d w}{d\theta}  &= -\frac{1}{d_{0}\gamma_{0}\mu_{0}^{3}}(c_{0}g_{1}+g_{2})\epsilon+O(\epsilon^{2})
   :=H_{2}(\theta,r,w)\epsilon+O(\epsilon^{2}).
\end{aligned}
\end{eqnarray}
Clearly, the functions $H_{j}(\cdot,r,w)$ are periodic functions of period $2\pi$.
In order to find the periodic solutions in system \eqref{2D-FHN},
we apply  the averaging theory and compute the averaged function of order one $(G_{1}(r,w),G_{2}(r,w))$,
which is obtained by the following integrals
\begin{eqnarray*}
G_{j}(r,w)=\int_{0}^{2\pi} H_{j}(\theta,r,w)\,d\theta, \ \ \ \ j=1,2.
\end{eqnarray*}
These yields
\begin{eqnarray*}
G_{1}(r,w)\!\!\!&=&\!\!\!\frac{\pi \gamma_{0}r}{d_{0}^{2}\mu_{0}^{3}}
      \left(2(3u_{0}-1.1)w+d_{0}\delta_{0}^{2}\gamma_{1}\right),\\
G_{2}(r,w)\!\!\!&=&\!\!\!-\frac{2\pi}{d_{0}\gamma_{0}\mu_{0}^{3}}
\left(\frac{c_{0}}{d_{0}}(3u_{0}-1.1)\left(\frac{1}{2}r^{2}+\gamma_{0}^{2}w^{2}\right)-\delta_{0}\gamma_{1}w\right).
\end{eqnarray*}
Solving $G_{1}(r,w)=0$ and $G_{2}(r,w)=0$ in the set $r>0$ yields that
\begin{eqnarray*}
r=r_{*}^{+}=\frac{d_{0}\gamma_{0}|\gamma_{1}|}{|3u_{0}-1.1|}\sqrt{|\frac{c_{0}\gamma_{0}^{4}-2\delta_{0}}{2c_{0}}|},\ \ \ \ \
w_{*}=\frac{5d_{0}\gamma_{0}^{2}\gamma_{1}}{11-30u_{0}},
\end{eqnarray*}
where $11-30u_{0}\neq 0$ can be checked by the proof of Lemma \ref{lm-FHP}.
Note that  $2\delta_{0}-c_{0}\gamma_{0}^{4}=0$
if and only if $\gamma_{0}=1/5^{1/3}<1$.
Then by Lemma \ref{lm-FHP}, we have $2\delta_{0}-c_{0}\gamma_{0}^{4}\neq 0$.
This implies
\begin{eqnarray*}
{\rm det}\left(
\begin{aligned}
\frac{\partial G_{1}}{\partial r} & \frac{\partial G_{1}}{\partial w}\\
\frac{\partial G_{2}}{\partial r} & \frac{\partial G_{2}}{\partial w}
\end{aligned}
\right)|_{(r_{*},w_{*})}=\frac{2\pi^{2}(2\delta_{0}-c_{0}\gamma_{0}^{4})\gamma_{0}^{2}\gamma_{1}^{2}}{d_{0}^{2}\mu_{0}^{6}}\neq 0.
\end{eqnarray*}
Thus by Lemma \ref{lm-averag-1},
for sufficiently small $\epsilon$ there exists an isolated periodic solution $\varphi(\theta,\epsilon)$ of period $2\pi$ in \eqref{2D-FHN}
such that $\varphi(0,\epsilon)\to (r_{*}^{+},w_{*})$ as $\epsilon\to 0$.
Consequently, by \eqref{chge-1} and \eqref{3D-FHN-2}, the first statement holds.
The second one is obtained by Theorem \ref{thm-stab-PTW}.
Therefore, the proof is now complete.
\end{proof}

\section{Concluding remarks}

We have studied the existence and stability of small-amplitude periodic traveling wave solutions
emerging from fold-Hopf equilibria in a general system of one reaction-diffusion equation coupled with one ordinary differential equation.
This system includes the FitzHugh-Nagumo system, caricature calcium models and other models in the real-world applications.
Under the assumption that the three-dimensional traveling wave system possesses a fold-Hopf equilibrium,
we transform it into a two-dimensional nonautonomous system,
and then apply the recent results on the averaging theory to
prove the existence of periodic solutions emerging from the fold-Hopf bifurcations.
Tsai, Zhang, Kirk and Sneyd \cite{Tsai-etal-2012} recently studied a simplified model of calcium dynamics (see \eqref{calcium-pde-0})
and numerically found that small-amplitude periodic waves  from fold-Hopf bifurcations are spectrally unstable.
In the present paper, we give a rigorous proof for the instability of these perturbed periodic waves.
We also point out  that the instability of these waves
 does not depend on whether fold-Hopf bifurcations are subcritical or supercritical
(see Theorem \ref{thm-stab-PTW}).

The preceding proof demonstrates the instability of  periodic waves from fold-Hopf bifurcation
by considering the perturbation problem for a single eigenvalue of the zero-amplitude case,
that is, the perturbation of $\lambda=A^{2}\mu_{0}^{2}$ for the linear operator $\mathcal{F}_{0}$.
This process only shows that there exists a certain unstable eigenvalue
for the related linearized operator associated with a perturbed periodic wave.
However, the spectrum of the linearization about a periodic waves could have a spectral bands contained in the unstable
half plane of complex values with positive real part (see \cite[Section 8.4]{Kapitula-Promislow-13}).
So it is also interesting to investigate the perturbation of spectral curves in the unstable half plane
and give the precise form of the unstable spectrum for the perturbed periodic waves.

Note that  we have only investigated periodic  waves arising from fold-Hopf bifurcations in the coupled system \eqref{PDE-ODE}.
As shown in the references \cite{Guck-Holmes-83,Kuznetsov-98},
there exist complex patterns arising near the fold-Hopf equilibria through bifurcation,
such as  heteroclinic orbits and torus.
So it is also interesting to further study the local stability of these perturbed patterns
from fold-Hopf bifurcation in the coupled system \eqref{PDE-ODE}.
By the similar method used in the present paper,
it is also possible to study the stability of small-amplitude periodic and quasi-periodic patterns arising in
other equations such as the classical and generalized Ginzburg-Landau equations \cite{Doelman-90,Duan-Holme-95}.

\section*{Acknowledgments}

The first author also thank Professor Haitao Xu from Huazhong University of Science and Technology for his valuable suggestions.
This work was partly supported by the National Natural Science Foundation of China (Grant No. 12101253),
the Scientific Research Foundation of CCNU (Grant No. 31101222044)
and the Science and Technology Project for Excellent Postdoctors of Hubei Province, China.

\section*{Author declarations}

\subsection*{Conflict of Interest}
The authors have no conflicts to disclose.

\subsection*{Data availability}
Data sharing is not applicable to this article as no data were created or analyzed in this study.

\bibliographystyle{amsplain}

\end{document}